\documentclass[toc=bib,DIV=15]{scrartcl}

\usepackage[T1]{fontenc}
\usepackage[utf8]{inputenc}
\usepackage[english]{babel}
\usepackage{lmodern}
\usepackage{exscale}
\usepackage[babel]{microtype}
\usepackage{mathtools, mathrsfs}
\usepackage{amssymb}
\usepackage{enumitem}
\usepackage{booktabs}
\usepackage{bbm}
\usepackage{tikz}
\usetikzlibrary{arrows.meta,positioning,decorations.pathreplacing,calc,fit,backgrounds,matrix}
\usepackage{tikz-cd}

\usepackage{csquotes}
\usepackage[style=alphabetic, sorting=nyt, maxnames=10, maxalphanames=5, backend=biber, giveninits=true]{biblatex}
\DeclareDelimFormat[textcite]{multinamedelim}{\textendash}
\DeclareDelimAlias[textcite]{finalnamedelim}{multinamedelim}
\DeclareNolabel{\nolabel{\regexp{[\p{Z}\p{P}\p{S}\p{C}]+}}}
\DeclareSourcemap{\maps[datatype=bibtex]{\map{\step[fieldset=issn, null]\step[fieldset=location, null]}}}

\usepackage[colorinlistoftodos]{todonotes}
\definecolor{softblue}{RGB}{60, 100, 210}
\DeclareDocumentCommand{\najib}{so+m}{\todo[color=green!70!black, \IfBooleanT{#1}{inline}, caption={\IfValueTF{#2}{#2}{#3}}]{#3}}
\DeclareDocumentCommand{\victor}{so+m}{\todo[color=softblue!80!white, \IfBooleanT{#1}{inline}, caption={\IfValueTF{#2}{#2}{#3}}]{#3}}

\definecolor{corail}{rgb}{0.9882,0.4627,0.4157}
\definecolor{viola}{RGB}{166,146,186}
\definecolor{mocha}{RGB}{164,120,100}
\definecolor{peachfuzz}{rgb}{0.8, 0.5451, 0.3961}
\definecolor{clouddancer}{RGB}{240, 238, 233}
\definecolor{bluefusion}{RGB}{73, 98, 117}
\definecolor{veiledvista2}{RGB}{109, 136, 110}
\usepackage[numbered]{bookmark}
\hypersetup{
  colorlinks,
  linkcolor=veiledvista2,
  citecolor=veiledvista2,
  urlcolor=veiledvista2,
}
\usepackage{amsthm}
\numberwithin{equation}{section}
\theoremstyle{plain}
\newtheorem{theorem}[equation]{Theorem}
\newtheorem{proposition}[equation]{Proposition}
\newtheorem{corollary}[equation]{Corollary}
\newtheorem{lemma}[equation]{Lemma}

\newtheorem{theoremintro}{Theorem}

\theoremstyle{definition}
\newtheorem{definition}[equation]{Definition}

\theoremstyle{remark}

\newtheorem{remark}[equation]{Remark}

\newcommand{\R}{\mathbb{R}}
\newcommand{\F}{\mathbb{F}}
\newcommand{\Z}{\mathbb{Z}}
\newcommand{\Q}{\mathbb{Q}}

\DeclareMathOperator{\Tors}{Tors}
\DeclareMathOperator{\tr}{tr}
\DeclareMathOperator{\Hom}{Hom}
\DeclareMathOperator{\Ext}{Ext}
\DeclareMathOperator{\ind}{ind}
\DeclareMathOperator{\gr}{gr}

\newcommand{\id}{\mathrm{id}}
\newcommand{\SL}{\mathrm{SL}}
\newcommand{\T}{\mathcal{T}}
\newcommand{\cC}{\mathcal{C}}
\newcommand{\cA}{\mathcal{A}}
\newcommand{\cU}{\mathcal{U}}

\newcommand{\HBM}{\bar{H}^{\mathrm{BM}}}

\begin{document}
\title{Odd-primary torsion in the homology of unordered configurations on the torus}
\author{Najib Idrissi\thanks{Université Paris Cité, Sorbonne Université, CNRS, IMJ-PRG, F-75013 Paris, France} \and Victor Roca i Lucio\thanks{Université Paris Cité, Sorbonne Université, CNRS, IMJ-PRG, F-75013 Paris, France}}
\date{August 25, 2026}
\maketitle

\begin{abstract}
  The main object of study of this paper is $B_k(\Sigma_1)$, the unordered configuration space of $k$ points in the torus $\Sigma_1$.
  First, we produce a marked-point transfer argument which, combined with Chen--Zhang's theorem, establishes that $H_*(B_k(\Sigma_1);\Z)$ has no $p$-torsion for $k\leq 2p-1$.
  Secondly, at the threshold $k=2p$, we prove that $H_{2p-2}(B_{2p}(\Sigma_1);\Z)$ has $p$-torsion if and only if $p\geq 5$.
  For $p\geq5$, the class is the image under puncture filling of a generator of the Bianchi--Stavrou torsion group $\Z/p$ of the once-punctured torus; for $p=3$, Napolitano's calculation shows that this punctured class dies after filling.
  Finally, we also prove that $H_{2p}(B_{2p}(\Sigma_1);\Z)$ and $H_{2p+1}(B_{2p}(\Sigma_1);\Z)$ have no $p$-torsion for every odd $p$.
\end{abstract}

\tableofcontents

\addsec{Introduction}
\label{sec:introduction}

Let $M$ be a manifold and let $B_k(M)$ denote the unordered configuration space of $k$ points in $M$.
When $M$ is a surface other than $S^2$ or $\mathbb{RP}^2$, its ordered and unordered configuration spaces are aspherical; hence $B_k(M)$ is a classifying space for the corresponding surface braid group, and the homology of $B_k(M)$ computes the group homology of that braid group~\cite{fadellConfigurationSpaces1962}.
For the plane, this homology has been completely understood since the 1970s: Fuks computed the mod-$2$ cohomology of the classical braid groups~\cite{fuksCohomologiesGroupCOS1970}, F.~Cohen determined the mod-$p$ homology at all primes~\cite{cohenHomologyIteratedLoop1976}, and Vainshtein described the integral cohomology~\cite{vainsteinCohomologyBraidGroups1978}.
For braid groups of closed surfaces, integral results remain far scarcer.
In this article, we study the homology of the torus braid groups, or equivalently the homology of the unordered configuration spaces of $k$ points on the torus $\Sigma_1$.
The group $\pi_1(B_k(\Sigma_1))$ is also called the \textit{elliptic} braid group; its group algebra supplies the braid-group input in constructions of type-$A$ double affine Hecke algebras~\cite{jordanQuantumDModules2009}.

With field coefficients, more is known about these spaces.
The graded $\mathbb F_2$-homology of unordered torus configurations is known from Bödigheimer--Cohen--Taylor~\cite{bodigheimerHomologyConfigurationSpaces1989}, and Zhang later described its Dyer--Lashof structure~\cite{zhangQuillenHomology2025}.
Rational Betti numbers were computed by Maguire, Schiessl, and Drummond-Cole--Knudsen~\cite{maguireComputingCohomology2016,schiesslBettiNumbersTorus2016,drummond-coleBettiNumbersConfiguration2017}; Pagaria determined the rational cohomology ring, mixed Hodge structure, and mapping-class-group action~\cite{pagariaCohomologyRingsTorus2020}.
Integrally, Napolitano computed the cohomology of $B_k(\Sigma_1)$ for $k\leq6$, and in that range all torsion is $2$-primary~\cite[Table~2]{napolitanoCohomologyConfigurationSpaces2003}.
The odd-primary integral homology, by contrast, remains largely unknown: the calculations just cited exhibit no odd-primary torsion class in $H_*(B_k(\Sigma_1);\Z)$ and do not decide whether such torsion exists beyond their ranges. The existence of odd-primary torsion was in fact predicted in \cite[Remark~3.7 in the arXiv version]{chenModHomologyUnordered2024}, but this question remained unsolved.

This scarcity has structural sources. Integral torsion is invisible to rational models, and calculations at odd primes involve genuinely characteristic-$p$ phenomena such as Bocksteins and power operations~\cite{chenModHomologyUnordered2024,bianchiHomologyConfigurationSpaces2024,zhangQuillenHomology2025}.
Moreover, the most effective geometric tools are specific to open surfaces: there, configurations can be pushed towards the missing boundary, which produces stabilization and scanning maps~\cite{mcduffConfigurationSpacesPositive1975} and underlies Bianchi--Stavrou's recent computation of the mod-$p$ homology of configuration spaces of compact orientable surfaces with one boundary component~\cite{bianchiHomologyConfigurationSpaces2024}.
A closed surface admits no such stabilization, and the puncture-filling map comparing the punctured and closed situations can kill torsion classes---and we prove that this actually happens on the torus at the prime~$3$.
Finally, the sought-after torsion lies beyond the reach of homological stability, which fixes the homological degree as $k$ tends to infinity~\cite{randalWilliamsHomologicalStability2013}: already for the classical braid groups and an odd prime $p$, the first $p$-torsion occurs at weight $2p$~\cite{vainsteinCohomologyBraidGroups1978}, and the torsion class studied here lies in the diagonal regime $k=2p$, in degree $2p-2$, with both the weight and the degree growing with $p$.

Our main results settle this threshold question for the torus.
For every prime $p\geq5$, we show that the first $p$-torsion in $H_*(B_k(\Sigma_1);\Z)$ appears exactly at weight $k=2p$, where we exhibit it in degree $2p-2$; for every odd prime, Chen--Zhang's theorem together with Theorem~\ref{thmA} gives no $p$-torsion below this threshold weight, and Theorem~\ref{thmC} excludes it in degrees $2p$ and $2p+1$ at the threshold itself.
The resulting dichotomy is a genuinely closed-surface phenomenon: the punctured torus carries a candidate torsion group $\Z/p$ for every odd prime, and it is the filling of the puncture that kills its generators when $p=3$ while preserving them for all $p\geq5$.

Beyond the torus case, and in a somewhat different direction, our recent work \cite{idrissiRocaPointSet2026} provides us with explicit point-set chain complexes over $\mathbb{F}_p$ which compute the mod-$p$ homology of unordered configuration spaces of parallelizable manifolds (which include the torus). These point-set constructions allowed us to perform computations in low weights and degrees \cite[Section~3.4 and Proposition~3.37]{idrissiRocaPointSet2026} using a dedicated computer package \cite{idrissiRocaLucioSoftware}. However, at present the computational threshold is easily reached, which explains why we considered different methods in this paper.

Let us also mention that analogous torsion questions are open for graph configuration spaces, whose homology is studied in particular as a model for quantum statistics of particles constrained to networks~\cite{maciazekNonAbelianQuantum2019}.
Ko--Park determine the torsion in first homology of unordered graph configurations~\cite{koCharacteristicsGraphBraid2012}; torsion-freeness of ordered configuration-space homology is known for restricted graph classes, including trees with loops by Chettih--Lütgehetmann~\cite{chettihHomologyConfiguration2018}, while the general assertion remains a folklore conjecture and no odd-primary torsion is presently known in unordered graph-configuration homology~\cite{anAsymptoticHomology2022,anAsymptoticCorrigendum2026,hainautRepresentationAsymptotics2025}.

\subsection*{Main results}

Chen--Zhang~\cite[Theorem~1.1]{chenModHomologyUnordered2024} proved that no $p$-torsion appears in $H_*(B_k(\Sigma_1);\Z)$ for $k\le p$ and odd $p$, and explicitly predicted in the arXiv version~\cite[Remark~3.7]{chenModHomologyUnordered2024} that the first $p$-power torsion should occur at weight $2p$.
Bianchi--Stavrou's punctured-surface calculation shows that the $p$-primary torsion subgroup at that weight is $\Z/p$~\cite{bianchiHomologyConfigurationSpaces2024}.
We prove that its generators survive puncture filling for every $p\ge5$, whereas Napolitano's calculation shows that they die for $p=3$~\cite{napolitanoCohomologyConfigurationSpaces2003}.
We also determine a high-degree non-torsion range at weight $2p$.
Thus the predicted threshold fails at $p=3$: there is no $3$-torsion through weight $6$, so the first such torsion, if any, occurs at weight at least $7$.
To the best of our knowledge, the surviving classes form the first infinite family of odd-primary torsion classes in unordered configuration spaces on a closed torus.

\begin{theoremintro}\label{thmA}
  For every odd prime $p$, if $p < k < 2p$, then $H_*(B_k(\Sigma_1);\Z)$ has no $p$-torsion.
\end{theoremintro}

\begin{theoremintro}\label{thmB}
  For every odd prime $p$, $H_{2p-2}(B_{2p}(\Sigma_1);\Z)$ has $p$-torsion if and only if $p \geq 5$.
  For $p\geq5$, puncture filling sends every generator of the $p$-primary torsion subgroup $\Z/p$ of $H_{2p-2}(B_{2p}(\Sigma_{1,1}^\circ);\Z)$ to a non-zero order-$p$ class on the closed torus.
\end{theoremintro}

\begin{theoremintro}\label{thmC}
  For every odd prime $p$, $H_{2p}(B_{2p}(\Sigma_1);\Z)$ and $H_{2p+1}(B_{2p}(\Sigma_1);\Z)$ have no $p$-torsion.
\end{theoremintro}

\subsection*{Proof strategies}

We prove Theorem~\ref{thmA} using the $k$-sheeted covering
$B_k^\odot(\Sigma_1)\to B_k(\Sigma_1)$, where $B_k^\odot(\Sigma_1)$ is the
space of unordered configurations with a marked point. For $p<k<2p$, one has
$p\nmid k$, so transfer makes the homology of $B_k(\Sigma_1)$ with
$\Z_{(p)}$-coefficients a direct summand of that of $B_k^\odot(\Sigma_1)$.
Translation in the torus gives a homeomorphism
\begin{equation*}
  B_k^\odot(\Sigma_1)
  \cong \Sigma_1\times B_{k-1}(\Sigma_{1,1}^\circ).
\end{equation*}
Bianchi--Stavrou~\cite{bianchiHomologyConfigurationSpaces2024} show that the
homology of the punctured factor is $p$-torsion-free in this range.

The proof of Theorem~\ref{thmB} is the bulk of the article. The key idea is to
consider the Gysin sequence associated to the puncture-filling inclusion
$j:B_k(\Sigma_{1,1}^\circ)\hookrightarrow B_k(\Sigma_1)$. With
$\Z_{(p)}$-coefficients, the relevant part is
\begin{equation*}
  H_{2p-3}(B_{2p-1}(\Sigma_{1,1}^\circ))
  \xrightarrow{\partial}
  H_{2p-2}(B_{2p}(\Sigma_{1,1}^\circ))
  \xrightarrow{j_*}
  H_{2p-2}(B_{2p}(\Sigma_1)).
\end{equation*}
This sequence relates the punctured and closed homology groups in the critical
weight. Concretely, Bianchi--Stavrou~\cite{bianchiHomologyConfigurationSpaces2024}
show that the $p$-primary torsion of
$H_{2p-2}(B_{2p}(\Sigma_{1,1}^\circ);\Z_{(p)})$ is cyclic of order $p$; we
write $\tau_p$ for a fixed generator and show, for $p\geq5$, that
$j_*(\tau_p)\neq0$.

Equivalently, for $p\geq5$, we show that $\tau_p \notin \operatorname{im}(\partial)$ by considering the natural group actions on both the source and the target of $\partial$, with respect to which the map $\partial$ is equivariant.
After reduction modulo $p$, the strategy is to find a finite equivariance group for which the target line is fixed and then to show that the source has no trivial composition factor; an equivariant image of the source cannot contain the fixed target class.
Let $\Gamma_{1,1}=\operatorname{Mod}(\Sigma_{1,1},\partial\Sigma_{1,1})$ be the mapping class group of the once-bordered torus $\Sigma_{1,1}$ relative to its boundary, let $H=H_1(\Sigma_{1,1};\F_p)$ be the standard $\SL_2(\F_p)$-module, and let $\T_{1,1}(p)=\ker(\Gamma_{1,1}\to\SL_2(\F_p))$ be the mod-$p$ Torelli group.
The Johnson filtration refers to the residual Torelli action on the two-step nilpotent quotient of the surface group; Bianchi--Stavrou's mod-$p$ version is the homomorphism~\cite[Definition~5.8 and Corollary~5.12]{bianchiHomologyConfigurationSpaces2024}
\begin{equation*}
  \xi_\tau^p:\T_{1,1}(p)\longrightarrow\Hom_{\F_p}(H,\Lambda^2H).
\end{equation*}
Its equivariance makes $K^\xi=\ker(\xi_\tau^p)$ normal in $\Gamma_{1,1}$~\cite[Lemma~5.14]{bianchiHomologyConfigurationSpaces2024}, and Bianchi--Stavrou's action calculation shows that $K^\xi$ acts trivially on their mod-$p$ cellular model~\cite[Propositions~4.19 and~5.22]{bianchiHomologyConfigurationSpaces2024}.
Consequently, the boundary is equivariant for the finite quotient $\widetilde G=\Gamma_{1,1}/K^\xi$, which fits into an exact sequence
\begin{equation*}
  1\longrightarrow A\longrightarrow\widetilde G\longrightarrow \SL_2(\F_p)\longrightarrow1.
\end{equation*}
In the particular case of the torus, $A=\operatorname{im}(\xi_\tau^p)$ is either $0$ or isomorphic to $H$. See Sections~\ref{sec:bs-action} and~\ref{sec:equivariance} for details.

Write $\beta_0$ for the non-zero reduction of $\tau_p$ modulo $p$, as in~\eqref{eq:beta0}, and apply equivariance to the reduced boundary $\bar\partial$. For every odd prime $p$, the line $\F_p\{\beta_0\}$ is the trivial one-dimensional representation of $\widetilde G$. For $p\geq5$, the source of $\bar\partial$ has no trivial composition factor. Hence $\beta_0$ cannot lie in the image of $\bar\partial$, which implies that the integral class $\tau_p$ does not lie in the image of $\partial$. For $p=3$, the source does contain trivial composition factors, so there is no such obstruction; Napolitano's computation~\cite{napolitanoCohomologyConfigurationSpaces2003} shows that the integral class lies in the image of $\partial$ and that $j_*(\tau_3)=0$.

Finally, Theorem~\ref{thmC} is proved using Napolitano's Borel--Moore cell complex~\cite{napolitanoCohomologyConfigurationSpaces2003} and the connecting map between the closed and punctured complexes.

\subsection*{Outlook}

We do not produce a complete $p$-local homology computation of $B_{2p}(\Sigma_1)$: we do not know whether there is further $p$-torsion in degree $2p-2$ for $p\geq5$, or whether there is $p$-torsion in degree $2p-1$.

The proofs suggest extensions to higher weights and higher genus.
The marked-point transfer applies only while the covering degree is a $p$-local unit and the punctured factor lies below the Bianchi--Stavrou threshold.

At larger weights, the source of the filling boundary $\partial$ contains additional representation-theoretic families, and trivial composition factors may appear. The exceptional prime $p=3$ already illustrates this limitation at the threshold: the target line remains a trivial $\widetilde G$-representation, but the source contains trivial composition factors, so equivariance no longer separates the target class from the image of $\bar\partial$. Napolitano's computation determines that the class is in fact hit.

% \victor{Dans l'intro, j'ai confondus $\tau_p$ et $\beta_0$ parce que de mon point de vue, le changement de coefficients de $\Z_{(p)}$ à $\F_p$ est un détail technique mineur. T'en penses quoi ? À la rigeur on pourrait mettre une petite footnote en disant qu'en réalité il y a un changement de coefs implicite.}
% \najib{D'Ok, maintenant j'utilise seulement $\tau_p$, en précisant une fois que le même symbole désigne sa réduction modulo $p$. Plus tard on distingue explicitement la classe intégrale $\tau_p$ de sa réduction $\beta_0$.}

In higher genus, the natural equivariance group retains the mod-$p$ Johnson layer in a finite extension
\begin{equation*}
  1\longrightarrow A_g\longrightarrow\widetilde G_g
  \longrightarrow\operatorname{Sp}_{2g}(\F_p)\longrightarrow1,
\end{equation*}
where $A_g$ is the image of the mod-$p$ Johnson homomorphism in $\Hom_{\F_p}(H,\Lambda^2H)$ for $H=H_1(\Sigma_{g,1};\F_p)$.
Unlike in genus one, this image need not vanish or be the standard module; in particular, the $\Lambda^3H$ part is potentially visible and must be controlled in the Johnson-layer extension and target-character arguments.
A higher-genus obstruction must therefore be formulated in terms of $\widetilde G_g$-composition factors, with the symplectic calculation appearing only on the associated graded, rather than in terms of $\operatorname{Sp}_{2g}(\F_p)$ alone.

\subsection*{Acknowledgements}

The authors acknowledge support from the project ANR-22-CE40-0008 SHoCoS, the IdEx University of Paris ANR-18-IDEX-0001, and the Institut Universitaire de France (IUF).
We thank Adela Zhang for pointing us towards this problem and for helpful discussions, Andrea Bianchi and Andreas Stavrou for their confirmation of Remark~\ref{rem:johnson-image}, and Noé Cuneo for assistance.

Large language models supplied by Anthropic and OpenAI were used during the mathematical work for literature discovery and brainstorming, and during the writing for language editing.
An LLM also drew attention to the issue discussed in Remark~\ref{rem:johnson-image}.
Every cited source was checked against the original, and all mathematical claims and proofs were independently rederived and verified by the authors.
The authors take full responsibility for the manuscript.

\section{Background}
\label{sec:inputs}

This section fixes the notation and records the torsion results used throughout the paper.
We also recall the two cellular models and the representation-theoretic facts that enter the proofs of Theorems~\ref{thmB} and~\ref{thmC}.

\subsection{Notation}
\label{sec:notation}

We establish the coefficient, configuration-space, and surface notation used in the rest of the paper.

Throughout this paper, $p$ is an odd prime, $\Z_{(p)}$ denotes the localization of $\Z$ at $p$, and $\F_p = \Z_{(p)}/p$ denotes the field with $p$ elements.
A finitely generated abelian group $M$ has no $p$-torsion if and only if $M\otimes\Z_{(p)}$ is a free $\Z_{(p)}$-module.
Homology without specified coefficients is integral homology, and $\HBM_*$ denotes the Borel--Moore homology of a space.
For a chain complex $C$, we write $H_*(C)$ for its homology.
For a finitely generated abelian group $M$, let $\Tors_p M$ denote its $p$-primary torsion subgroup.

Since $\Z_{(p)}$ is flat over $\Z$, there is for every space $X$ a natural isomorphism $H_i(X;\Z)\otimes_{\Z}\Z_{(p)}\cong H_i(X;\Z_{(p)})$.
When $H_i(X;\Z)$ is finitely generated, it restricts to a natural isomorphism
\begin{equation}
  \Tors_p H_i(X;\Z)
  \xrightarrow{\ \sim\ }
  \Tors H_i(X;\Z_{(p)}),
  \label{eq:localization-torsion}
\end{equation}
onto the torsion submodule of the localized group.
We use~\eqref{eq:localization-torsion} throughout without further comment, and we write the same symbol for a $p$-primary torsion class and for its image in localized homology; naturality makes this compatible with all the induced maps used below.

For a space $X$ and an integer $k$ called the \emph{weight}, we set:
\begin{equation*}
  B_k(X) \coloneqq
  \{(x_1,\ldots,x_k)\in X^k : x_i\neq x_j\text{ for }i\neq j\} \bigm/ \mathfrak S_k.
\end{equation*}

We let $\Sigma_1$ denote the closed torus.
We fix throughout a closed disk $D \subset \Sigma_1$ and let $q \in D$ be its center point.
We let $\Sigma_{1,1} = \Sigma_1 \setminus \mathring{D}$ denote a torus with one boundary component, and let $\Sigma_{1,1}^\circ$ denote its interior, a once-punctured torus; thus $\Sigma_{1,1}^\circ = \Sigma_1\setminus D$.
When only the homeomorphism type is relevant, we identify $\Sigma_{1,1}^\circ$ with the homeomorphic deleted-point torus $\Sigma_1\setminus\{q\}$.
The puncture-filling inclusion is denoted by
\begin{equation*}
  j:B_k(\Sigma_{1,1}^\circ) \hookrightarrow B_k(\Sigma_1).
\end{equation*}

For each fixed weight, the cellular models recalled in Sections~\ref{sec:bs-action} and~\ref{sec:napolitano} have finitely many cells.
Consequently, all integral homology groups of the configuration spaces used below are finitely generated.
We use this when applying the structure theorem over $\Z_{(p)}$ and the universal coefficient theorem.

% \victor{Est-ce qu'on enlève cette notation et on garde simplement: $B_k(\Sigma_1)$, $B_k(\Sigma_{1,1}^\circ)$, et $B_k^\odot(\Sigma_1)$ ?}

\subsection{Known torsion results}
\label{sec:known-torsion}

We collect the known results on $p$-torsion in the homology of $B_k(\Sigma_1)$ and
$B_k(\Sigma_{1,1}^\circ)$, notably those of
Napolitano~\cite{napolitanoCohomologyConfigurationSpaces2003},
Bianchi--Stavrou~\cite{bianchiHomologyConfigurationSpaces2024}, and
Chen--Zhang~\cite{chenModHomologyUnordered2024}.
We then extract the integral punctured-torus torsion group at weight
$2p$ and explain why the two adjacent mod-$p$ excess classes form a single
Bockstein pair.
% \victor{Est-ce qu'on fait appel à un moment au Bockstein de $\beta_0$ ? Je suis pas sûr. Si c'est pas le cas, je vois pas trop l'intérêt de l'introduire.}
% \najib{On n'utilise pas $\alpha_1$ dans l'obstruction elle-même. Elle intervient dans le calcul de l'homologie à coefficients entiers. Elle explique l'excès mod $p$ en degré $2p-1$ et permet de conclure que $t_{2p-1}=0$. Elle marque aussi le seuil $2p$ dans la comparaison bar--CE.}

\begin{proposition}[{Napolitano~\cite[Table~2]{napolitanoCohomologyConfigurationSpaces2003}}]
  \label{prop:napolitano}
  One has
  \begin{equation}
    H^5(B_6(\Sigma_1);\Z)\cong\Z^9\oplus(\Z/2)^8.
    \label{eq:napolitano-b6}
  \end{equation}
  Consequently, by the universal coefficient theorem, $H_4(B_6(\Sigma_1);\Z)$ has no $3$-torsion.
\end{proposition}

\begin{theorem}[{Chen--Zhang~\cite[Theorem~1.1]{chenModHomologyUnordered2024}}]
  \label{thm:chen-zhang}
  If $k\leq p$, then $H_*(B_k(\Sigma_1);\Z)$ has no $p$-torsion.
\end{theorem}

\begin{theorem}[{Consequence of Bianchi--Stavrou~\cite[Corollaries~C.5--C.7]{bianchiHomologyConfigurationSpaces2024}}]
  \label{thm:punctured-threshold}
  For every $k\leq2p$,
  \begin{equation*}
    \Tors_p H_i(B_k(\Sigma_{1,1}^\circ);\Z)
    \cong
    \begin{cases}
      \Z/p, & \text{if } (k,i)=(2p,2p-2),\\
      0, & \text{otherwise}.
    \end{cases}
  \end{equation*}
\end{theorem}

By Theorem~\ref{thm:punctured-threshold}, the $p$-primary torsion subgroup of
$H_{2p-2}(B_{2p}(\Sigma_{1,1}^\circ);\Z)$ is cyclic of order $p$.
It has $p-1$ generators; let us fix one of them once and for all, written
\begin{equation}
  \tau_p\in H_{2p-2}(B_{2p}(\Sigma_{1,1}^\circ);\Z_{(p)})
  \label{eq:tau-p}
\end{equation}
under the identification~\eqref{eq:localization-torsion}, and let us write
\begin{equation}
  \beta_0=\overline{\tau_p}
  \in H_{2p-2}(B_{2p}(\Sigma_{1,1}^\circ);\F_p)
  \label{eq:beta0}
\end{equation}
for its non-zero reduction.
Nothing below depends on this choice.  Indeed, the other generators are the
$\lambda\tau_p$ with $\lambda\in(\Z/p)^\times$, and if $f$ is any homomorphism
defined on the group in~\eqref{eq:tau-p}, then $p\tau_p=0$ forces
$pf(\tau_p)=0$, so $f(\tau_p)$ lies in an $\F_p$-vector space and
$f(\lambda\tau_p)=\lambda f(\tau_p)$ vanishes if and only if $f(\tau_p)$ does.
Thus the assertions $j_*(\tau_p)=0$ and $\beta_0\neq0$ used below hold for one
generator if and only if they hold for all of them.

Bianchi--Stavrou's Bockstein theorem~\cite[Corollary~C.5]{bianchiHomologyConfigurationSpaces2024}
shows that the $p$-primary torsion in the integral homology of
$B_{2p}(\Sigma_{1,1}^\circ)$ has exponent $p$.
We also denote the torsion and excess dimensions of $B_{2p}(\Sigma_{1,1}^\circ)$ by:
\begin{equation*}
  t_i \coloneqq
  \dim_{\F_p}\operatorname{Tors}_pH_i(B_{2p}(\Sigma_{1,1}^\circ);\Z),
  \qquad
  e_i \coloneqq
  \dim_{\F_p}H_i(B_{2p}(\Sigma_{1,1}^\circ);\F_p)
  -\dim_{\Q}H_i(B_{2p}(\Sigma_{1,1}^\circ);\Q).
\end{equation*}
Then universal coefficients give $e_i=t_i+t_{i-1}$.

\begin{proof}[Derivation from Bianchi--Stavrou]
  We recall the extraction because the bookkeeping is used later.
  Since
  $1\le p-2$ for all odd $p$, Bianchi--Stavrou's fake-basechange result~\cite[Corollary~C.6]{bianchiHomologyConfigurationSpaces2024} applies to genus $1$.
  In the notation of that corollary, there is an isomorphism of weighted
  bigraded $\F_p$-vector spaces
  \begin{equation*}
    H_*(B_\bullet(\Sigma_{1,1}^\circ);\F_p)
    \cong
    \bigl(H_*(B_\bullet(\Sigma_{1,1}^\circ);\Q)\bigr)^{\F_p}
    \otimes_{\F_p}
    \F_p[\beta_0,\alpha_1,\beta_1,\alpha_2,\beta_2,\ldots],
  \end{equation*}
  where the superscript denotes a bigraded $\F_p$-vector space having in each
  bidegree the same dimension as the rational factor.
  The polynomial factor has generators
  \begin{equation*}
    \alpha_i\in H_{2p^i-1}(B_{2p^i}(\Sigma_{1,1}^\circ);\F_p),
    \quad i\geq1,
    \qquad
    \beta_i\in H_{2p^{i+1}-2}(B_{2p^{i+1}}(\Sigma_{1,1}^\circ);\F_p),
    \quad i\geq0.
  \end{equation*}
  Thus the only nonconstant polynomial generators of weight $\le2p$ are
  $\beta_0$, in degree $2p-2$, and $\alpha_1$, in degree $2p-1$.
  No nonconstant product occurs in weight $\le2p$.

  For $k<2p$, only the constant of the polynomial factor contributes, so the
  mod-$p$ Betti numbers equal the rational Betti numbers in every degree and
  universal coefficients force no $p$-torsion.
  At $k=2p$, each of $\beta_0$ and $\alpha_1$ occurs once, tensored with the
  one-dimensional weight-zero, degree-zero part of the rational factor.
  Hence $e_{2p-2}=e_{2p-1}=1$ and $e_i=0$ for
  $i\notin\{2p-2,2p-1\}$.
  Bianchi--Stavrou's low-degree no-torsion result~\cite[Corollary~C.7]{bianchiHomologyConfigurationSpaces2024} gives $t_i=0$ for $i<2p-2$.
  Hence $t_{2p-2}=1$, then $t_{2p-1}=0$, and all higher $t_i$ vanish by the excess equation and the fact that $B_k(\Sigma_{1,1}^\circ)$ has homological dimension $k$.
  Since the torsion has exponent $p$, the unique torsion group is $\Z/p$.
\end{proof}

The mod-$p$ classes $\beta_0$ and $\alpha_1$ should be read as a Bockstein
pair, not as two independent integral features. The integral feature is the
single summand $\Z/p\{\tau_p\}$ in degree $2p-2$. The mod-$p$ class
$\beta_0$ is its reduction, and the mod-$p$ class $\alpha_1$ in degree
$2p-1$ is the universal-coefficient shadow of the same torsion group. Thus
$t_{2p-2}=1$ and $t_{2p-1}=0$, even though the mod-$p$ excess appears in
both adjacent degrees.

\subsection{The Bianchi--Stavrou equivariant action model}
\label{sec:bs-action}

We recall the part of Bianchi--Stavrou's cellular model needed for the
equivariant obstruction in the proof of Theorem~\ref{thmB}.
The model first identifies the punctured-torus cellular chains with a two-sided
bar complex, then makes the mapping-class action explicit on its surface
factor.  The quadratic part of this action defines a finite Johnson quotient;
filtering it out later reduces the source calculation to the standard
$\SL_2(\F_p)$-action.
For $r\geq1$, let $\mathcal R_r=\bigvee_{i=1}^rS^1$, based at the wedge point.
The weighted \emph{unordered Moriyama} algebra
$\operatorname{UMor}_\bullet(r)$ is assembled from the reduced cellular
chains of the one-point compactifications of
$B_n(\mathcal R_r\setminus\{*\})$, in weight $-n$~\cite[Definition~3.4]{bianchiHomologyConfigurationSpaces2024}.
Its differential is zero~\cite[Proposition~3.3]{bianchiHomologyConfigurationSpaces2024},
and superposition of configurations gives an
isomorphism of weighted rings~\cite[Proposition~3.9]{bianchiHomologyConfigurationSpaces2024}
\begin{equation*}
  \operatorname{UMor}_\bullet(r)
  \cong
  \Lambda_{\Z}(x_1,\ldots,x_r)
  \otimes
  \Gamma_{\Z}(y_1,\ldots,y_r),
\end{equation*}
where $x_i$ represents one point on the $i$th open arc and $y_i^{[m]}$
represents $2m$ points there.  Thus $x_i$ has weight $-1$ and
$y_i^{[m]}$ has weight $-2m$; their regraded degrees are equal to their
weights.

We now specialize to the once-bordered torus.  Put
\begin{equation*}
  \Gamma_{1,1} \coloneqq
  \operatorname{Mod}(\Sigma_{1,1},\partial\Sigma_{1,1}),
  \qquad
  G \coloneqq
  \pi_1(\Sigma_{1,1})=\langle\gamma_a,\gamma_b\rangle,
  \qquad
  H_{\Z} \coloneqq
  H_1(\Sigma_{1,1};\Z).
\end{equation*}
Bianchi--Stavrou stratify the one-point compactification of $B_n(\Sigma_{1,1}^\circ)$ by
records $t=(b,P,v)$.  Such a record describes $b$ vertical bars in a
rectangle, with $P_i\geq1$ points on the $i$th bar, and $v_a,v_b\geq0$
points on the two open arcs of a bouquet representing
$\gamma_a,\gamma_b$~\cite[Definitions~4.1--4.3 and Proposition~4.4]{bianchiHomologyConfigurationSpaces2024}.
The corresponding cell has Borel--Moore degree $n+b$, where
$\sum_iP_i+v_a+v_b=n$.  It is placed in weight $-n$ and in regraded
homological degree $(n+b)-2n=b-n$~\cite[Notation~4.5]{bianchiHomologyConfigurationSpaces2024}.
Let
\begin{equation}\label{eq:bs-cellular-complex}
  \cC_{\Z}
  \coloneqq
  \bigoplus_{n\geq0}
  \widetilde C_*^{\mathrm{cell}}
  \bigl((B_n(\Sigma_{1,1}^\circ))^+;\Z\bigr),
\end{equation}
where the displayed chain degree is this regraded degree.

Let
\begin{equation*}
  \cA\coloneqq\operatorname{UMor}_\bullet(1)
  =\Lambda_{\Z}(x)\otimes\Gamma_{\Z}(y),
  \qquad
  \cU\coloneqq\operatorname{UMor}_\bullet(2)
  =\Lambda_{\Z}(x_a,x_b)\otimes\Gamma_{\Z}(y_a,y_b).
\end{equation*}
The boundary word $\zeta=[\gamma_a,\gamma_b]\in G$ induces a ring homomorphism
$\cA\to\cU$ and hence makes $\cU$ a left
$\cA$-module.  In genus one it is given by~\cite[Definition~4.12,
Lemma~4.13, and Section~6.1]{bianchiHomologyConfigurationSpaces2024}
\begin{equation}
  x\longmapsto0,
  \qquad
  y\longmapsto2x_ax_b,
  \qquad
  y^{[m]}\longmapsto0\quad(m\geq2).
  \label{eq:boundary-word-action}
\end{equation}
This is the genus-one specialization of the general assignment
$y^{[m]}\mapsto\Omega_{2m}$: here $\Omega_2=2x_ax_b$, while
$\Omega_{2m}\in\Lambda^{2m}H_{\Z}$ vanishes for $m\geq2$ because
$\Lambda^{>2}H_{\Z}=0$ in genus one.
Bianchi--Stavrou identify the cellular complex with the reduced two-sided bar
complex~\cite[Proposition~6.1]{bianchiHomologyConfigurationSpaces2024}
\begin{equation}
  \cC_{\Z}
  \cong
  \overline B_\bullet(\Z,\cA,\cU),
  \qquad
  \overline B_b(\Z,\cA,\cU)
  \coloneqq
  \Z\otimes\bar\cA^{\otimes b}\otimes\cU,
  \label{eq:bs-bar-complex}
\end{equation}
where $\bar\cA$ is the augmentation ideal; a term of bar length $b$
and weight $-n$ has regraded chain degree $b-n$.  If
$P_i=2m_i+\epsilon_i$ and $v_j=2m'_j+\epsilon'_j$, with
$\epsilon_i,\epsilon'_j\in\{0,1\}$, then $t=(b,P,v)$ corresponds to
\begin{equation*}
  e_{(b,P,v)}
  \longleftrightarrow
  1\otimes
  x^{\epsilon_1}y^{[m_1]}\otimes\cdots\otimes
  x^{\epsilon_b}y^{[m_b]}\otimes
  \prod_{j\in\{a,b\}}x_j^{\epsilon'_j}y_j^{[m'_j]}.
\end{equation*}
The initial bar face vanishes.  The other faces merge adjacent bar entries
or apply~\eqref{eq:boundary-word-action} to the last entry before multiplying
the surface factor.  Geometrically these are, respectively, collisions of
adjacent bars and the rightmost bar reaching the
bouquet~\cite[Proposition~4.14]{bianchiHomologyConfigurationSpaces2024}.

Bianchi--Stavrou's cellular approximation gives a $\Gamma_{1,1}$-action by chain
maps on $\cC_{\Z}$.  Under~\eqref{eq:bs-bar-complex}, it fixes every
bar entry and acts on $\cU$ through the induced action on
$G$~\cite[Propositions~4.18 and~4.19]{bianchiHomologyConfigurationSpaces2024}.

To describe the action on $\cU$, write $[w]\in H_{\Z}$ for the
abelianization of $w\in G$, and let
\begin{equation*}
  (-)_x:H_{\Z}\longrightarrow\Z\{x_a,x_b\},
  \qquad
  (-)_y:H_{\Z}\longrightarrow\Z\{y_a,y_b\}
\end{equation*}
send $[\gamma_i]$ to $x_i$ and $y_i$, respectively, for $i\in\{a,b\}$,
and extend $(-)_x$ to exterior powers.  The \emph{content map} is
\begin{equation*}
  c:\Z[G]\longrightarrow\Lambda H_{\Z}.
\end{equation*}
It is the $\Z$-linear extension of the multiplicative map on $G$ determined by
$c(\gamma_i)=1+[\gamma_i]$ and
$c(\gamma_i^{-1})=1-[\gamma_i]$, with products taken in the order of the
letters in a word.  This is consistent with cancellation because
$(1+[\gamma_i])(1-[\gamma_i])=1$ in $\Lambda H_{\Z}$.
Let $c_2(w)\in\Lambda^2H_{\Z}$ be the exterior-degree-two component of
$c(w)$.  For $\phi\in\Gamma_{1,1}$, Bianchi--Stavrou's action formula on
$\cU$ is~\cite[Theorem~3.13]{bianchiHomologyConfigurationSpaces2024}
\begin{equation}
  x_i\longmapsto[\phi(\gamma_i)]_x,
  \qquad
  y_i\longmapsto[\phi(\gamma_i)]_y+[c_2(\phi(\gamma_i))]_x.
  \label{eq:bs-action}
\end{equation}
Here and below $i\in\{a,b\}$.  The second summand in the image of $y_i$ is
the quadratic correction term.  Although $[\phi(\gamma_i)]_x$ and
$[\phi(\gamma_i)]_y$ have the same coefficients in the respective bases,
they lie in distinct linear subspaces of $\cU$.
The action is by ring automorphisms and also determines every divided power:
since $m!y_i^{[m]}=y_i^m$ and $\cU$ is torsion-free,
$\phi(y_i^{[m]})$ is the unique element whose product with $m!$ is
$\phi(y_i)^m$.  Put $\cU_{\F_p}=\cU\otimes\F_p$, and use the same symbols
$x_i,y_i$ for the reductions of the integral generators.

Throughout the article, we set:
\begin{equation*}
  H \coloneqq
  H_{\Z}\otimes\F_p \cong H_1(\Sigma_{1,1};\F_p),
  \quad
  \T_{1,1}(p)
  \coloneqq
  \ker\bigl(\Gamma_{1,1}\longrightarrow\SL(H)\bigr).
\end{equation*}
The group $\T_{1,1}(p)$ is the mod-$p$ Torelli subgroup.  The quadratic corrections define the
homomorphism~\cite[Definition~5.8 and Corollary~5.12]{bianchiHomologyConfigurationSpaces2024}
\begin{equation*}
  \xi_\tau^p:\T_{1,1}(p)
  \longrightarrow
  \Hom_{\F_p}(H,\Lambda^2H),
  \qquad
  \xi_\tau^p(\phi)([\gamma_i])=c_2(\phi(\gamma_i))\bmod p.
\end{equation*}
Since $\xi_\tau^p$ takes values in a mod-$p$ group, the next display is an
identity in $\cU_{\F_p}$, and $(-)_x$ there denotes the reduction modulo
$p$ of the map introduced above.  For $\phi\in\T_{1,1}(p)$ one has
$[\phi(\gamma_i)]\equiv[\gamma_i]\bmod p$, so formula~\eqref{eq:bs-action}
reduces to
\begin{equation*}
  \phi(x_i)=x_i,
  \qquad
  \phi(y_i)=y_i+[\xi_\tau^p(\phi)([\gamma_i])]_x
  \qquad\text{in }\cU_{\F_p}.
\end{equation*}
The corresponding integral statement in $\cU$ is false: a mod-$p$ Torelli
element need only fix $x_i$ modulo $p$.
Bianchi--Stavrou prove that $\xi_\tau^p$ is
$\Gamma_{1,1}$-equivariant~\cite[Lemma~5.14]{bianchiHomologyConfigurationSpaces2024},
so
\begin{equation*}
  K^\xi\coloneqq\ker(\xi_\tau^p)
\end{equation*}
is normal in $\Gamma_{1,1}$.  Their proof of Proposition~5.22 checks all divided
powers $y_i^{[m]}$, including $m\geq p$, and shows that $K^\xi$ acts
trivially on $\cU_{\F_p}$.  Since the bar entries are fixed,
$K^\xi$ acts trivially on the entire mod-$p$ cellular complex.

\subsection{Napolitano's cellular decomposition}
\label{sec:napolitano}

Napolitano~\cite{napolitanoCohomologyConfigurationSpaces2003} constructs a Borel--Moore cellular decomposition for the homology of $B_k(\Sigma_1)$ which generalizes the Fuks--Vainshtein~\cite{fuksCohomologiesGroupCOS1970,vainsteinCohomologyBraidGroups1978} model for the plane.
We will use this model to prove Theorem~\ref{thmC}.
Its index records precisely the data controlling the bottom Borel--Moore cells
and their closed--punctured boundary incidences.
Let us now recall it.

Fix a base point on the circle, e.g., the north pole.
The \emph{index} of a configuration $\xi \in B_k(S^1)$ is $\ind(\xi) \coloneqq (k-v, v)$, where $v = 1$ if $\xi$ contains the base point, and $v = 0$ otherwise.

Now, consider the projection $\pi : S^1 \times \R \to \R$.
Given an element $\xi \in B_k(S^1 \times \R)$, we can consider the projection $\pi(\xi)$, which consists of $d$ ordered points; the fiber over any of these points is an unordered configuration $\xi_i$ of $k_i$ points in $S^1$, with $\sum k_i = k$.
We can then define $\ind(\xi) \coloneqq (\ind(\xi_1), \ldots, \ind(\xi_d))$.
This defines a cellular decomposition of $(B_k(S^1 \times \R))^+$, with boundaries as indicated in~\cite[Section~2.3]{napolitanoCohomologyConfigurationSpaces2003}.

Finally, the torus $\Sigma_1$ is the union of a circle $S^1$ and a cylinder $S^1 \times \R$, so a configuration $\xi$ on $\Sigma_1$ is the union of a configuration $\xi'$ on the circle and a configuration $\xi''$ on the cylinder.
The \emph{index} of $\xi$ is
\begin{equation}\label{eq:index}
  \ind(\xi) \coloneqq (\ind(\xi'), \ind(\xi'')) = (e, v; (m_1, v_1), \dots, (m_d, v_d)),
\end{equation}
where $v_i \in \{0,1\}$, $m_i \in \mathbb{N}$, $m_i+v_i \geq 1$, and $e + v + \sum m_i + \sum v_i = k$.
This, again, defines a cellular decomposition of $(B_k(\Sigma_1))^+$, with boundaries as indicated in~\cite[Section~2.3]{napolitanoCohomologyConfigurationSpaces2003}.
The Borel--Moore degree of the cell associated to the index in Equation~\eqref{eq:index} is $n = e + \sum_i (m_i+1)$.
We will use it in Section~\ref{sec:high-degree}, where we will also fix orientations and signs for this cellular decomposition.

\begin{remark}
  Both Napolitano and Bianchi--Stavrou use an ordered-bar stratification of the same Fuks--Vainshtein type.
  Napolitano's model~\cite[Sections~2.3--2.4]{napolitanoCohomologyConfigurationSpaces2003} treats the closed surface directly and retains the boundary faces arising from pinned points, the bouquet vertex, and the two ends of the cylinder.
  In the Bianchi--Stavrou model~\cite[Propositions~4.9 and~4.14]{bianchiHomologyConfigurationSpaces2024}, the adjacent-bar mergers have the same signed-shuffle incidence coefficients (the parity-binomial numbers denoted $P(a,b)$ in Section~\ref{sec:high-degree}) while the remaining surface contribution is packaged in the $\operatorname{UMor}_\bullet(2)$-factor and a single term determined by the symplectic class.
  This reorganization makes the bar description and the mapping-class-group action explicit.
  We do not identify the two complexes term by term: below we use Napolitano's model for the closed-torus calculation and the Bianchi--Stavrou model for the punctured, equivariant calculation.
\end{remark}

\subsection{Representation theory of \texorpdfstring{$\SL_2(\F_p)$}{SL\_2(F\_p)}}
\label{sec:rep-conventions}

We record the representation-theoretic conventions and the two elementary facts used later to separate the source and target of the mod-$p$ Gysin map.
We refer to~\cite{humphreysRepresentationsSl21975} for the representation theory of $\SL_2(\F_p)$.
As before, let $H = H_1(\Sigma_{1,1}; \F_p)$ be the standard two-dimensional representation of $\SL_2(\F_p)$.
For $0\leq r\leq p-1$, let the $r$th divided power of $H$ (of dimension $r+1$) be:
\begin{equation*}
  V_r \coloneqq \Gamma^r H.
\end{equation*}
For $0\leq r\leq p-1$, the canonical map from symmetric tensors to symmetric coinvariants is an $\SL_2(\F_p)$-equivariant isomorphism ($r!$ is invertible in
$\F_p$).
Thus $V_r\cong\operatorname{Sym}^r H$ is the simple module of highest weight $r$.
In particular, $V_0$ is trivial and $V_{p-1}$ is the Steinberg module.
Write $\omega$ for the nondegenerate $\SL_2(\F_p)$-invariant intersection pairing on $H$.
The map $v\mapsto\omega(v,-)$ identifies $H$ equivariantly with $H^\vee$, and hence
$V_r^\vee\cong V_r$ in this range.

For $1\leq r\leq p-2$, the modular Clebsch--Gordan formula gives:
\begin{equation}
  [H\otimes\Gamma^rH]=[\Gamma^{r+1}H]+[\Gamma^{r-1}H]
  \label{eq:clebsch-gordan}
\end{equation}
in the Grothendieck group; see~\cite[Lemma~1.3]{DotyHenke2005} (where $\Gamma^r H = L(r)$).

We shall also use two standard facts about modular representations.
If $P\trianglelefteq E$ is a normal $p$-subgroup of a finite group, then $P$ acts trivially on every simple $\F_p[E]$-module; equivalently, every such simple module is inflated from $E/P$~\cite[Section~8.3, Corollary to Proposition~26]{serreLinearRepresentations1977}.
Here \emph{inflated} means that an $E/P$-action is pulled back along the
quotient map $E\twoheadrightarrow E/P$, so that $P$ acts trivially.
Moreover, if $V$ and $W$ are finite-dimensional $E$-modules over a field and
$\phi:V\to W$ is $E$-linear, then the absence of a trivial composition factor
in $V$ implies that $\operatorname{im}\phi$ contains no non-zero $E$-fixed
vector: the image is a quotient of $V$, while a fixed vector spans a trivial
submodule.

\section{Marked-point transfer and Theorem~\ref{thmA}}
\label{sec:proof-thma}

Let us now prove Theorem~\ref{thmA} using the marked-point transfer.
The first lemma splits the homology of the closed configuration space off the
marked cover, and the second identifies that cover with a product whose
punctured factor is below the Bianchi--Stavrou torsion threshold.

\begin{definition}
  Let $X$ be a topological space. The marked-point unordered configuration space of $X$ is
  \begin{equation*}
    B_k^\odot(X)
    \coloneqq
    \{(Q,x)\in B_k(X)\times X:x\in Q\}.
  \end{equation*}
\end{definition}

\begin{lemma}
  \label{lem:direct-summand}
  If $p\nmid k$, then $H_*(B_k(\Sigma_1);\Z_{(p)})$ is a direct summand of $H_*(B_k^\odot(\Sigma_1);\Z_{(p)})$.
\end{lemma}

\begin{proof}
  The forgetful map $\pi:B_k^\odot(\Sigma_1)\to B_k(\Sigma_1)$ is a $k$-sheeted covering.
  Its transfer satisfies
  \begin{equation*}
    \pi_*\circ\tr=k \cdot \id_{H_*(B_k(\Sigma_1);\Z_{(p)})}.
  \end{equation*}
  Since $k$ is a unit in $\Z_{(p)}$, the rescaled transfer splits $\pi_*$.
\end{proof}

\begin{lemma}
  \label{lem:split-marked}
  There is a homeomorphism
  \begin{equation*}
    B_k^\odot(\Sigma_1)\cong \Sigma_1 \times B_{k-1}(\Sigma_{1,1}^\circ).
  \end{equation*}
\end{lemma}

\begin{proof}
  Regard $\Sigma_1$ as an abelian topological group with identity element $0$.
  The maps
  \begin{align*}
    \Sigma_1 \times B_{k-1}(\Sigma_1\setminus\{0\})
    &\longrightarrow B_k^\odot(\Sigma_1),
    & (x,R)&\longmapsto(\{x\}\cup(x+R),x),\\
    B_k^\odot(\Sigma_1)
    &\longrightarrow \Sigma_1 \times B_{k-1}(\Sigma_1\setminus\{0\}),
    & (Q,x)&\longmapsto(x,(Q-x)\setminus\{0\})
  \end{align*}
  are mutually inverse homeomorphisms.
  Composing with a homeomorphism $\Sigma_1\setminus\{0\}\cong\Sigma_{1,1}^\circ$, which exists since both spaces are once-punctured tori, gives the result.
\end{proof}

\begin{proof}[Proof of Theorem~\ref{thmA}]
  Suppose $p<k<2p$.
  Since the homology is finitely generated, Theorem~\ref{thm:punctured-threshold} implies that $H_*(B_{k-1}(\Sigma_{1,1}^\circ);\Z_{(p)})$ is free because $k-1<2p$.
  By Lemma~\ref{lem:split-marked} and the Künneth theorem over $\Z_{(p)}$, $H_*(B_k^\odot(\Sigma_1);\Z_{(p)})$ is free.
  Since $p\nmid k$, Lemma~\ref{lem:direct-summand} makes $H_*(B_k(\Sigma_1);\Z_{(p)})$ a direct summand of a free $\Z_{(p)}$-module, hence free.
  Localization therefore shows that $H_*(B_k(\Sigma_1);\Z)$ has no $p$-primary torsion.
\end{proof}

Together with Theorem~\ref{thm:chen-zhang}, Theorem~\ref{thmA} gives the
combined range $k\leq2p-1$ stated in the abstract.

\section{The Gysin exact sequence}
\label{sec:gysin}

The puncture-filling inclusion gives a Gysin sequence whose connecting map controls whether the punctured torsion class survives on the closed torus.
After writing this sequence, we reduce survival to a mod-$p$ equivariant separation problem and identify the action on its target line.

Recall the choices $q \in D \subset \Sigma_1$ and the identification $\Sigma_{1,1} = \Sigma_1\setminus \mathring{D}$, so that $\Sigma_{1,1}^\circ = \Sigma_1\setminus D$.
Inside the unordered configuration space $B_{2p}(\Sigma_1)$, the configurations containing the point $q$ form a closed submanifold of codimension $2$; deleting $q$ from such a configuration identifies it with a configuration space of the \emph{deleted-point} torus,
\begin{equation}
  \{Q\in B_{2p}(\Sigma_1):q\in Q\}\cong B_{2p-1}(\Sigma_1\setminus\{q\}).
  \label{eq:puncture-submanifold}
\end{equation}
Its open complement is $B_{2p}(\Sigma_1\setminus\{q\})$, and its normal bundle is the trivial oriented plane bundle with fiber $T_q\Sigma_1$.

The standard collar inclusions
$B_k(\Sigma_{1,1}^\circ)\hookrightarrow B_k(\Sigma_1\setminus\{q\})$
are homotopy equivalences whose composites with
$B_k(\Sigma_1\setminus\{q\})\hookrightarrow B_k(\Sigma_1)$ are the
puncture-filling maps $j$.
The Gysin exact sequence of~\eqref{eq:puncture-submanifold} therefore becomes
\begin{equation}
  \cdots \to
  H_{2p-3}(B_{2p-1}(\Sigma_{1,1}^\circ);\Z_{(p)})
  \xrightarrow{\partial}
  H_{2p-2}(B_{2p}(\Sigma_{1,1}^\circ);\Z_{(p)})
  \xrightarrow{j_*}
  H_{2p-2}(B_{2p}(\Sigma_1);\Z_{(p)})
  \to \cdots
  \label{eq:gysin}
\end{equation}

\subsection{Reduction modulo \texorpdfstring{$p$}{p}}
\label{sec:mod-p-reduction}

We convert the integral survival question into the assertion that $\beta_0$ is not in the image of the mod-$p$ connecting map.
Since $B_{2p-1}(\Sigma_{1,1}^\circ)$ lies below the punctured threshold, its homology with $\Z_{(p)}$-coefficients is free; this gives the coefficient-change identification for the domain of $\partial$ used below.
The codomain need not be torsion-free: it contains the order-$p$ class $\tau_p$, whose reduction modulo $p$ is non-zero.

Let $\tau_p$ be the generator of the order-$p$ torsion subgroup of
$H_{2p-2}(B_{2p}(\Sigma_{1,1}^\circ);\Z_{(p)})$ fixed in~\eqref{eq:tau-p}.
By exactness,
\begin{equation}
  j_*(\tau_p)\neq0
  \quad\Longleftrightarrow\quad
  \tau_p\notin\operatorname{im}\partial.
  \label{eq:gysin-survival}
\end{equation}

The universal coefficient theorem over $\Z_{(p)}$ gives a natural short exact sequence
\begin{equation}
  0\longrightarrow
  H_i(Y;\Z_{(p)})\otimes_{\Z_{(p)}}\F_p
  \longrightarrow H_i(Y;\F_p)
  \longrightarrow
  \operatorname{Tor}_1^{\Z_{(p)}}
  \bigl(H_{i-1}(Y;\Z_{(p)}),\F_p\bigr)
  \longrightarrow0.
  \label{eq:uct-mod-p}
\end{equation}
Since $2p-1<2p$, Theorem~\ref{thm:punctured-threshold}
implies that $H_*(B_{2p-1}(\Sigma_{1,1}^\circ);\Z_{(p)})$ is free.
Applying~\eqref{eq:uct-mod-p}
in degree $2p-3$ therefore yields a natural
identification
\begin{equation}
  M\coloneqq H_{2p-3}(B_{2p-1}(\Sigma_{1,1}^\circ);\F_p)
  \cong
  H_{2p-3}(B_{2p-1}(\Sigma_{1,1}^\circ);\Z_{(p)})
  \otimes_{\Z_{(p)}}\F_p.
  \label{eq:mod-p-source}
\end{equation}
Set
\begin{equation*}
  N\coloneqq
  H_{2p-2}(B_{2p}(\Sigma_{1,1}^\circ);\Z_{(p)})
  \otimes_{\Z_{(p)}}\F_p
\end{equation*}
and let $\bar\partial=\partial\otimes_{\Z_{(p)}}\F_p:M\to N$.
Under the natural coefficient-change injection
\[
  N\hookrightarrow H_{2p-2}(B_{2p}(\Sigma_{1,1}^\circ);\F_p),
\]
the element $\overline{\tau_p}\in N$ maps to the class $\beta_0$
of~\eqref{eq:beta0} by naturality of coefficient change.
Indeed, finite generation and Theorem~\ref{thm:punctured-threshold} give a
decomposition
\begin{equation*}
  H_{2p-2}(B_{2p}(\Sigma_{1,1}^\circ);\Z_{(p)})
  \cong \Z_{(p)}^r\oplus\Z_{(p)}/(p)\{\tau_p\}
\end{equation*}
for some $r$.  Thus $\tau_p$ does not lie in $p$ times this homology group, so
its image in $N$ is non-zero; the coefficient-change injection identifies that
image with $\beta_0$.

The reduction maps fit into the commutative square
\begin{equation}
  \begin{tikzcd}[column sep=large]
    H_{2p-3}(B_{2p-1}(\Sigma_{1,1}^\circ);\Z_{(p)})
    \arrow[r,"\partial"]
    \arrow[d,"x\mapsto \overline{x}"']
    & H_{2p-2}(B_{2p}(\Sigma_{1,1}^\circ);\Z_{(p)})
    \arrow[d,"y\mapsto \overline{y}"] \\
    M \arrow[r,"\bar\partial"'] & N
  \end{tikzcd}
  \label{eq:mod-p-gysin-square}
\end{equation}
We use only the commutativity of this square; we do not need tensoring~\eqref{eq:gysin}
with $\F_p$ to preserve exactness.

\begin{lemma}
  \label{lem:mod-p-reduction}
  If $\beta_0\notin\operatorname{im}\bar\partial$, then $0\neq j_*(\tau_p)\in H_{2p-2}(B_{2p}(\Sigma_1);\Z_{(p)})$.
\end{lemma}

\begin{proof}
  Suppose, to the contrary, that $j_*(\tau_p)=0$.
  Exactness of~\eqref{eq:gysin} then gives an element
  $x\in H_{2p-3}(B_{2p-1}(\Sigma_{1,1}^\circ);\Z_{(p)})$ such that
  $\partial x=\tau_p$.
  Commutativity of~\eqref{eq:mod-p-gysin-square} now gives
  \begin{equation*}
    \bar\partial(\overline{x})
    =\overline{\partial x}
    =\overline{\tau_p}
    =\beta_0,
  \end{equation*}
  contradicting $\beta_0\notin\operatorname{im}\bar\partial$.
  Hence $j_*(\tau_p)\neq0$.
\end{proof}

Theorem~\ref{thmB} therefore reduces to showing that
$\beta_0$ is not in the image of $\bar\partial$.

\subsection{Mapping-class equivariance and the target line}
\label{sec:equivariance}

We place the mod-$p$ Gysin map in the category of modules over the finite Johnson quotient $\widetilde G$.
We first show that the target line spanned by $\beta_0$ is a trivial $\widetilde G$-representation for every odd prime.
The restriction $p\geq5$ enters only later, when we exclude trivial composition factors from the source.

Continue with the notation of Section~\ref{sec:bs-action}, and write
$H=H_1(\Sigma_{1,1};\F_p)\cong\F_p^2$.
Because $\Gamma_{1,1}$ fixes the boundary pointwise, it should not be confused with
$\SL_2(\Z)$: its integral homological action surjects onto $\SL_2(\Z)$, but
the Dehn twist about the boundary component lies in the kernel. Set
\begin{equation}
  A=\operatorname{im}(\xi_\tau^p),
  \qquad
  \widetilde G=\Gamma_{1,1}/K^\xi.
  \label{eq:johnson-quotient}
\end{equation}
Reduction modulo $p$ of the homological action gives a surjection
$\Gamma_{1,1}\twoheadrightarrow\SL(H)=\SL_2(\F_p)$~\cite[Section~2.2]{farbPrimerMappingClass2012}.
Since $K^\xi\subset\T_{1,1}(p)$ and
$\T_{1,1}(p)/K^\xi\cong A$, there is an exact sequence
\begin{equation}
  1\longrightarrow A\longrightarrow\widetilde G
  \longrightarrow\SL_2(\F_p)\longrightarrow1.
  \label{eq:finite-extension}
\end{equation}
Thus $\widetilde G$ is finite and $A$ is an elementary abelian $p$-group.

Extending boundary-fixing diffeomorphisms across $D$ by the identity preserves
the submanifold~\eqref{eq:puncture-submanifold} and its oriented normal bundle.
The collar inclusions used to write~\eqref{eq:gysin} are equivariant under
such extensions.
Hence the Gysin sequence~\eqref{eq:gysin} and its reduction $\bar\partial$ are
$\Gamma_{1,1}$-equivariant. By Section~\ref{sec:bs-action}, $K^\xi$ acts trivially on
the mod-$p$ cellular complex and therefore on $M$. It also acts trivially on
$N$, by equivariance of the coefficient-change injection
$N\hookrightarrow H_{2p-2}(B_{2p}(\Sigma_{1,1}^\circ);\F_p)$. Thus $M$ and $N$ are
$\F_p[\widetilde G]$-modules and $\bar\partial:M\to N$ is
$\widetilde G$-linear.

\begin{lemma}
  \label{lem:johnson-image}
  As an $\SL_2(\F_p)$-module, $A$ is either $0$ or the standard module $H$.
\end{lemma}

\begin{remark}\label{rem:johnson-image}
  Bianchi--Stavrou's Proposition~5.17~\cite{bianchiHomologyConfigurationSpaces2024} would give
  $A\cong\Lambda^3H=0$ in our notation, but the cited proof of the generation statement used there does not establish the genus-one case. More precisely, it uses
  \cite[Theorem~5.2]{bianchiHomologyConfigurationSpaces2024}, attributed to Cooper~\cite{cooperTwoModP2015} and Perron~\cite{perronFiltrationJohnsonGroupe2008}, which states that $\T_{g,1}(p)$ is generated by $\T_{g,1}$ and the $p$th powers of all Dehn twists.
  Perron's note contains no proof but points to a theorem of Bass--Milnor--Serre~\cite{bassSolutionCongruence1967}.
  Cooper~\cite[Lemma~5.1]{cooperTwoModP2015} gave an argument for the corresponding statement which invokes a normal-generation result stated as~\cite[Theorem~5.4]{cooperTwoModP2015} and attributed to Bass--Milnor--Serre~\cite{bassSolutionCongruence1967}.
  The underlying symplectic theorem is stated only in rank at least two~\cite[Theorem~12.4]{bassSolutionCongruence1967}; hence this proof does not cover genus one.
  In fact, the unrestricted statement is false in genus one for $p\geq7$.

  Indeed, project the symplectic representation of $\Gamma_{1,1}$ to $\mathrm{PSL}_2(\Z)$, and let $T$ denote the image of a non-separating Dehn twist.
  The ordinary Torelli group maps trivially.
  The subgroup generated by the $p$th powers of all Dehn twists is normal, since
  $fD_c^pf^{-1}=D_{f(c)}^p$ for every mapping class $f$ and Dehn twist $D_c$.
  Separating twists act trivially on homology and all non-separating twists are
  conjugate, so its image in $\mathrm{PSL}_2(\Z)$ is precisely the normal closure
  $\langle\!\langle T^p\rangle\!\rangle$.
  For odd $p$, the projective image of $\T_{1,1}(p)$ is precisely the projective principal congruence subgroup
  \begin{equation*}
    \overline\Gamma(p)
    =\ker\bigl(\mathrm{PSL}_2(\Z)\longrightarrow
    \mathrm{PSL}_2(\F_p)\bigr).
  \end{equation*}
  Indeed, a projective class trivial modulo $p$ has a representative congruent
  to either $I$ or $-I$, and changing its sign if necessary gives a representative
  congruent to $I$.
  Thus, if \cite[Theorem~5.2]{bianchiHomologyConfigurationSpaces2024} held in genus one,
  this normal closure would equal $\overline\Gamma(p)$.
  However, the standard presentation
  $\mathrm{PSL}_2(\Z)=\langle s,r\mid s^2=r^3=1\rangle$, with $T=sr$,
  gives
  \begin{equation*}
    \mathrm{PSL}_2(\Z)/\langle\!\langle T^p\rangle\!\rangle
    \cong \mathrm{D}(2,3,p) = \langle s,r\mid s^2=r^3=(sr)^p=1\rangle,
  \end{equation*}
  where $\mathrm{D}(2,3,p)$ is the orientation-preserving triangle group (or von Dyck group).
  This triangle group on the right is hyperbolic, and hence infinite, for
  $p\geq7$, since $\frac12+\frac13+\frac1p<1$, whereas
  $\mathrm{PSL}_2(\Z)/\overline\Gamma(p)\cong\mathrm{PSL}_2(\F_p)$
  is finite.
  Thus the two subgroups cannot coincide.

  The Bianchi--Stavrou inputs used here are independent of this generation
  statement.  Proposition~4.19 precedes it; Corollary~5.12 and Lemma~5.14 are
  proved from Proposition~5.11; and the implication from Proposition~5.22 that
  $\ker(\xi_\tau^p)$ acts trivially on
  $\operatorname{UMor}_\bullet(2)\otimes\F_p$ is proved directly from the ring
  action and its values on all divided powers.  Only the further description
  of this kernel by Dehn-twist generators passes through Proposition~5.17 and
  Theorem~5.2, and we do not use that description.
  Lemma~\ref{lem:johnson-image} uses only the equivariance of $\xi_\tau^p$ to
  obtain the dichotomy $A=0$ or $A=H$, and either case suffices below.
\end{remark}

\begin{proof}[Proof of Lemma~\ref{lem:johnson-image}]
  In genus $1$, the target of $\xi_\tau^p$ is
  \begin{equation*}
    \Hom(H, \Lambda^2 H) \cong H^\vee \otimes \det(H)\cong H.
  \end{equation*}
  The second isomorphism uses the symplectic self-duality of $H$ and the triviality of the determinant on
  $\SL_2(\F_p)$.
  By~\cite[Lemma~5.14]{bianchiHomologyConfigurationSpaces2024}, $A$ is an $\SL_2(\F_p)$-submodule of this standard module.
  The standard module is irreducible, so $A=0$ or $A=H$.
\end{proof}

\begin{lemma}
  \label{lem:perfect-quotient}
  If $p\geq5$, then $\widetilde G$ is perfect (i.e., its abelianization vanishes).
\end{lemma}

\begin{proof}
  Since $A$ is abelian, conjugation by $\widetilde G$ induces an $\SL_2(\F_p)$-action on $A$.
  The map
  $A\to\widetilde G^{\mathrm{ab}}$ kills every element $g\cdot a-a$ and
  therefore factors through the coinvariants $A_{\SL_2(\F_p)} = A/\langle g\cdot a-a\rangle_{g\in \SL_2(\F_p),\ a\in A}$.
  Moreover, the quotient map $\widetilde G\to \SL_2(\F_p)$ induces an exact sequence
  \begin{equation}
    A_{\SL_2(\F_p)}\longrightarrow \widetilde G^{\mathrm{ab}}
    \longrightarrow \SL_2(\F_p)^{\mathrm{ab}}\longrightarrow 0.
    \label{eq:abelianization-johnson-extension}
  \end{equation}

  We first reprove the classical fact that $\SL_2(\F_p)$ is perfect for $p \geq 5$.
  Choose $\lambda\in\F_p^\times$ with $\lambda^2\neq1$, and set
  \begin{equation*}
    d=
    \begin{pmatrix}\lambda&0\\0&\lambda^{-1}
    \end{pmatrix},\qquad
    u(t)=
    \begin{pmatrix}1&t\\0&1
    \end{pmatrix},\qquad
    \ell(t)=
    \begin{pmatrix}1&0\\t&1
    \end{pmatrix}.
  \end{equation*}
  Direct calculation gives
  \begin{equation*}
    [d,u(t)]=u\bigl((\lambda^2-1)t\bigr),\qquad
    [d,\ell(t)]=\ell\bigl((\lambda^{-2}-1)t\bigr).
  \end{equation*}
  Both displayed coefficients are non-zero.
  Hence every upper and lower unipotent matrix is a commutator.
  These matrices generate $\SL_2(\F_p)$, so $\SL_2(\F_p)$ is perfect and $\SL_2(\F_p)^{\mathrm{ab}}=0$.

  It remains to treat the left-hand term.
  By Lemma~\ref{lem:johnson-image}, either $A=0$, in which case $A_{\SL_2(\F_p)}=0$, or $A=H$.
  In the latter case, let $e,f$ be the standard basis of $H$.
  Then
  \begin{equation*}
    (u(1)-1)f=e,\qquad (\ell(1)-1)e=f.
  \end{equation*}
  Thus the relations defining the coinvariants contain both basis vectors, and consequently $(H)_{\SL_2(\F_p)}=0$.
  Both outer terms in~\eqref{eq:abelianization-johnson-extension} therefore vanish, so $\widetilde G^{\mathrm{ab}}=0$.
  Equivalently, $\widetilde G$ is perfect.
\end{proof}

\begin{lemma}
  \label{lem:target-fixed}
  For every odd prime $p$, the line $\F_p\{\beta_0\}\subset N = H_{2p-2}(B_{2p}(\Sigma_{1,1}^\circ);\Z_{(p)}) \otimes_{\Z_{(p)}}\F_p$ is the trivial one-dimensional representation of $\widetilde G$.
\end{lemma}

\begin{proof}
  By Theorem~\ref{thm:punctured-threshold}, the $p$-primary torsion
  subgroup of $H_{2p-2}(B_{2p}(\Sigma_{1,1}^\circ);\Z_{(p)})$ is $T_p=\Z_{(p)}/(p)\{\tau_p\}$.
  This subgroup is characteristic.
  % Indeed, every mapping class acts on homology by an automorphism, and $p^r x=0$ implies $p^r(g\cdot x)=g\cdot(p^r x)=0$.
  Hence the $\Gamma_{1,1}$-action preserves $T_p$ and therefore also its
  non-zero image $\F_p\{\beta_0\}\subset N$ under the direct-sum decomposition
  established in Section~\ref{sec:mod-p-reduction}.
  Since the action on $N$ factors through $\widetilde G$, this is a $\widetilde G$-stable line.
  Let $\chi:\widetilde G\to\F_p^\times$ be its character.
  If $p\geq5$, then $\chi$ is trivial because $\widetilde G$ is perfect by Lemma~\ref{lem:perfect-quotient}.

  Suppose that $p=3$.
  Since $A$ is an elementary abelian $3$-group and $\F_3^\times$ has order $2$, the restriction of $\chi$ to $A$ is trivial.
  Thus $\chi$ factors through $\SL_2(\F_3)$.
  The upper and lower unipotent matrices $u(1)$ and $\ell(1)$ generate $\SL_2(\F_3)$ and both have order $3$.
  Their images under a character with values in $\F_3^\times$ are therefore trivial, so $\chi$ is trivial in this case as well.
  Thus every element of $\widetilde G$ fixes $\beta_0$, as claimed.
\end{proof}

\section{The source representation}
\label{sec:source-representation}

We compute the composition factors of the $\widetilde G$-module $M=H_{2p-3}(B_{2p-1}(\Sigma_{1,1}^\circ);\F_p)$, the source of the mod-$p$ Gysin map $\bar\partial:M\to N$.
We use a filtration to reduce the action to $\SL_2(\F_p)$ on the associated graded, identify the resulting complex with a small Chevalley--Eilenberg complex, and compute the required bidegree.

\subsection{The divided-power filtration}
\label{sec:divided-power-filtration}

We filter the Bianchi--Stavrou complex by divided-power degree in its surface variables.
On the associated graded, the quadratic correction term in~\eqref{eq:bs-action}
disappears, while the composition factors of the source homology are preserved.

Let $\cC=\cC_{\Z}\otimes\F_p$, where $\cC_{\Z}$ is the integral
Bianchi--Stavrou complex~\eqref{eq:bs-cellular-complex}, and write
$\cC^{(k)}$ for its summand of weight $-k$.  Each fixed weight is
finite-dimensional, and $\cC$ uses the regraded cellular degree, so
equivariant Poincar\'e--Lefschetz duality gives~\cite[Proposition~2.3]{bianchiHomologyConfigurationSpaces2024}
\begin{equation}
  H_q\bigl(\cC^{(k)}\bigr)
  \cong \HBM_{2k+q}(B_k(\Sigma_{1,1}^\circ);\F_p),
  \qquad
  H_q\bigl(\cC^{(k)}\bigr)^\vee
  \cong H_{-q}(B_k(\Sigma_{1,1}^\circ);\F_p).
  \label{eq:bs-duality-regrading}
\end{equation}
In particular, $M=H_{2p-3}(B_{2p-1}(\Sigma_{1,1}^\circ);\F_p)$ is dual to
$H_{-(2p-3)}(\cC^{(2p-1)})$, whose Borel--Moore degree is $2p+1$.

After reducing~\eqref{eq:bs-bar-complex} modulo $p$, a basis term has the form
\begin{equation*}
  1\otimes a_1\otimes\cdots\otimes a_b\otimes
  x_{i_1}\cdots x_{i_t}y_a^{[r]}y_b^{[s]},
\end{equation*}
where each $a_i=x^{\epsilon_i}y^{[m_i]}\neq1$ is a basis monomial of
$\bar\cA\otimes\F_p$.  Its \emph{surface divided-power degree} is $r+s$;
divided powers in the bar entries $a_i$ are not counted.  Let
$F_d\cC$ be spanned by terms of surface divided-power degree at most
$d$.  This is an increasing finite filtration in each fixed weight.  With
$F_{-1}\cC=0$, we write
\begin{equation*}
  \gr_F\cC=\bigoplus_{d\geq0} F_d\cC/F_{d-1}\cC
\end{equation*}
for the associated-graded chain complex.

\begin{lemma}
  \label{lem:filtered-action}
  The $\widetilde G$-action preserves the divided-power filtration on $\cC$, and its
  quadratic correction term in~\eqref{eq:bs-action} strictly lowers the surface divided-power degree.
  The associated-graded action factors through $\SL_2(\F_p)$ and is the
  standard linear action on both pairs $(x_a,x_b)$ and $(y_a,y_b)$.
\end{lemma}

\begin{proof}
  Fix $\phi\in\Gamma_{1,1}$ and $i\in\{a,b\}$.  Put
  $\widetilde L_i=[\phi(\gamma_i)]_y$ and
  $\widetilde z=x_ax_b\in\cU$.  Since $\Lambda^2H_{\Z}$ is free on
  $[\gamma_a]\wedge[\gamma_b]$, there is a unique $q_i\in\Z$ such that
  $[c_2(\phi(\gamma_i))]_x=q_i\widetilde z$.  Thus~\eqref{eq:bs-action}
  gives $\phi(y_i)=\widetilde L_i+q_i\widetilde z$ in $\cU$.
  Let $L_i\in\F_p\{y_a,y_b\}$, $\lambda_i\in\F_p$, and
  $z\in\cU_{\F_p}$ be the respective reductions modulo $p$ of
  $\widetilde L_i$, $q_i$, and $\widetilde z$.  Then
  $\phi(y_i)=L_i+\lambda_i z$ in $\cU_{\F_p}$.

  Since the action is by ring automorphisms and $m!y_i^{[m]}=y_i^m$, for
  $m\geq1$ we have in the torsion-free integral algebra $\cU$
  \begin{align*}
    m!\,\phi(y_i^{[m]})
    &=(\widetilde L_i+q_i\widetilde z)^m\\
    &=\widetilde L_i^m
    +mq_i\widetilde z\widetilde L_i^{m-1}\\
    &=m!\bigl(\widetilde L_i^{[m]}
    +q_i\widetilde z\widetilde L_i^{[m-1]}\bigr),
  \end{align*}
  where the middle equality uses $\widetilde z^2=0$.  Cancelling $m!$
  integrally and then
  reducing modulo $p$ gives, for $m\geq1$,
  \begin{equation*}
    \phi\bigl(y_i^{[m]}\bigr)
    =
    L_i^{[m]}+\lambda_izL_i^{[m-1]}.
  \end{equation*}
  In particular, this remains valid when $m\geq p$.  Over either coefficient
  ring,
  \begin{equation*}
    (\alpha y_a+\beta y_b)^{[m]}
    =
    \sum_{r+s=m}\alpha^r\beta^s y_a^{[r]}y_b^{[s]}.
  \end{equation*}
  Thus the linear part preserves surface divided-power degree, whereas each
  term involving the quadratic correction lowers it by one.  The action is
  multiplicative and fixes the bar entries, so it preserves
  $F_d\cC$.

  On the associated graded, all the lower-degree correction terms disappear.
  By~\eqref{eq:bs-action}, the remaining action is the mod-$p$ homological
  action on each of the two distinct linear spans of $(x_a,x_b)$ and
  $(y_a,y_b)$.
  Since $A$ is the kernel in~\eqref{eq:finite-extension}, it acts trivially,
  and the action factors through
  $\widetilde G/A\cong\SL_2(\F_p)$ as the standard linear action on both pairs.
\end{proof}

\begin{lemma}
  \label{lem:homology-associated-graded}
  The differential on $\cC$ is homogeneous of surface divided-power degree
  zero.  Consequently,
  \begin{equation}
    \gr_F H_*(\cC)\cong H_*(\gr_F\cC).
    \label{eq:homology-associated-graded}
  \end{equation}
\end{lemma}

\begin{proof}
  The bar differential is homogeneous in surface divided-power degree.
  Indeed, its internal faces leave the surface factor unchanged, while its
  final face multiplies that factor by the image of
  \eqref{eq:boundary-word-action}, which lies in
  $\Lambda(x_a,x_b)$.  Hence $\cC$ is the direct sum of its
  subcomplexes of exact surface divided-power degree.  Taking homology
  degree by degree gives~\eqref{eq:homology-associated-graded}.
\end{proof}

\begin{corollary}
  \label{cor:filtered-composition-factors}
  The filtration on $\cC$ induces a finite $\widetilde G$-stable filtration
  on
  \begin{equation*}
    M=H_{2p-3}(B_{2p-1}(\Sigma_{1,1}^\circ);\F_p).
  \end{equation*}
  The action on its associated graded $\gr_F M$ factors through
  $\SL_2(\F_p)$, and the Jordan--Hölder multiset of $M$ as a
  $\widetilde G$-module is obtained by inflating the Jordan--Hölder multiset
  of $\gr_F M$ along
  $\widetilde G\twoheadrightarrow\SL_2(\F_p)$.
\end{corollary}

\begin{proof}
  Passing to the cochain dual via~\eqref{eq:bs-duality-regrading} gives a
  $\widetilde G$-stable filtration on $M$.
  By Lemmas~\ref{lem:filtered-action}
  and~\ref{lem:homology-associated-graded}, its successive quotients are
  $\SL_2(\F_p)$-modules inflated along
  $\widetilde G\twoheadrightarrow\SL_2(\F_p)$.  The Jordan--Hölder multiset
  of a finite filtered module is the multiset union of those of its
  successive quotients, which proves the claim.
  % Unnecessary:
  % The filtration need not split $\widetilde G$-equivariantly: the quadratic
  % correction term may still change the extension data between successive pieces,
  % but it cannot change their simple factors.
\end{proof}

\subsection{The low-weight Chevalley--Eilenberg model and comparison}
\label{sec:source-input}

Lie and Chevalley--Eilenberg models for configuration-space homology arise from Knudsen's description via factorization homology, and their positive-characteristic low-weight forms are developed through spectral-Lie methods by Chen--Zhang and Zhang~\cite{knudsenBettiNumbersStability2017,chenModHomologyUnordered2024,zhangQuillenHomology2025}.
We identify the associated-graded Bianchi--Stavrou complex in weights below $2p$ with the corresponding small Chevalley--Eilenberg complex.
The resulting $\SL_2(\mathbb F_p)$-equivariant comparison turns the required source representation into a concrete bidegree calculation.

Since $\Sigma_{1,1}^\circ$ is connected and noncompact, its compactly supported
cohomology has no degree-zero class. Using the natural identification
$H_c^*(\Sigma_{1,1}^\circ;\F_p)\cong
H^*(\Sigma_{1,1},\partial\Sigma_{1,1};\F_p)$, choose a basis
\begin{equation*}
  H_c^*(\Sigma_{1,1}^\circ;\F_p)
  =\F_p\{a,b,c\},
  \qquad
  |a|=|b|=1,
  \qquad
  |c|=2,
\end{equation*}
with, up to graded commutativity, the single non-zero product
\begin{equation}
  a\smile b=c.
  \label{eq:punctured-cup-product}
\end{equation}
Via the intersection pairing, identify
$H_c^1(\Sigma_{1,1}^\circ;\F_p)$ with the standard module
$H=H_1(\Sigma_{1,1};\F_p)$, and use the same letters $a,b$ for the resulting
symplectic basis.
The class $c=a\smile b$ spans $\det(H)$ and is fixed by
$\SL_2(\F_p)$.

Let $\mathfrak l_{\mathrm{small}}$ be the two-step \emph{shifted} Lie algebra with
basis $x_2,w_2$, where $x_2$ has degree $2$ and weight $1$, $w_2$ has degree
$3$ and weight $2$, and
\begin{equation*}
  [x_2,x_2]=w_2,
  \qquad
  [x_2,w_2]=[w_2,w_2]=0.
\end{equation*}
For $p\geq5$, this is the free shifted Lie algebra on $x_2$.
For $p=3$, $\mathfrak l_{\mathrm{small}}$ denotes the explicitly defined
two-step algebra above; the comparison lemma below uses this presentation and
does not require it to be free on $x_2$.
Set
\begin{equation}
  \mathfrak g_{\mathrm{small}}
  =H_c^{-*}(\Sigma_{1,1}^\circ;\F_p)
  \otimes\mathfrak l_{\mathrm{small}}.
  \label{eq:punctured-shifted-lie-algebra}
\end{equation}

The superscript $-*$ records the regrading from cohomological to homological
degrees: a class in $H_c^j$ is placed in homological degree $-j$.
We suppress the tensor symbol in pure tensors, e.g., $ax_2 = a\otimes x_2$.
If $\theta\in H_c^j(\Sigma_{1,1}^\circ;\F_p)$ and $v$ is homogeneous in
$\mathfrak l_{\mathrm{small}}$, then
\begin{equation}
  \deg(\theta v)=|v|-j,
  \qquad
  \operatorname{weight}(\theta v)=\operatorname{weight}(v).
  \label{eq:decorated-class-degree}
\end{equation}
The six elements in Table~\ref{tab:punctured-ce-generators} form a homogeneous
basis of $\mathfrak g_{\mathrm{small}}$.

\begin{table}[htbp]
  \centering
  \begin{tabular}{@{}lllll@{}}
    \toprule
    Element & Weight & Degree & $\SL_2(\F_p)$-type & CE factor \\
    \midrule
    $E_0=cx_2$ & $1$ & $0$ & $\mathbf 1$ & $\Gamma$ (even) \\
    $O_a=ax_2,\ O_b=bx_2$ & $1$ & $1$ & $H$ & $\Lambda$ (odd) \\
    $O_2=cw_2$ & $2$ & $1$ & $\mathbf 1$ & $\Lambda$ (odd) \\
    $F_a=aw_2,\ F_b=bw_2$ & $2$ & $2$ & $H$ & $\Gamma$ (even) \\
    \bottomrule
  \end{tabular}
  \caption{The Chevalley--Eilenberg cogenerators of
  $\mathfrak g_{\mathrm{small}}$.}
  \label{tab:punctured-ce-generators}
\end{table}

The Chevalley--Eilenberg chain complex of
$\mathfrak g_{\mathrm{small}}$ is
\begin{equation}
  \mathcal E
  \coloneqq CE(\mathfrak g_{\mathrm{small}})
  =\Gamma(E_0,F_a,F_b)\otimes\Lambda(O_a,O_b,O_2).
  \label{eq:punctured-ce-algebra}
\end{equation}
We use the same symbols for the six basis elements of
$\mathfrak g_{\mathrm{small}}$ and for the corresponding Chevalley--Eilenberg
cogenerators.
Each odd cogenerator occurs at most once in a monomial, while an even cogenerator $z$
has divided powers $z^{[r]}$ for $r\geq0$.
Weight and degree are additive;
in particular,
\begin{equation*}
  F_a^{[i]}F_b^{[r-i]},
  \qquad 0\leq i\leq r,
\end{equation*}
form the standard basis of $\Gamma^r H$.

The cup product~\eqref{eq:punctured-cup-product} gives the only non-zero bracket between the elements in Table~\ref{tab:punctured-ce-generators}.
Thus
\begin{equation}
  [O_a,O_b]=O_2 \text{ in } \mathfrak{g}_{\mathrm{small}},
  \qquad
  d(O_aO_b)=O_2 \text{ in } \mathcal E,
  \label{eq:punctured-ce-differential}
\end{equation}
and the differential vanishes on each cogenerator.

We now compare $\mathcal E$ with the associated graded of the
Bianchi--Stavrou complex $\gr_F\cC$ introduced in Section~\ref{sec:divided-power-filtration}.

Write $\mathcal E^{(k)}$ for the weight-$k$ summand of $\mathcal E$.
For each weight $k$, define a chain complex $\mathcal D^{(k)}$ by
\begin{equation*}
  \mathcal D^{(k)}_d
  =\operatorname{Hom}_{\F_p}
  \bigl((\gr_F\cC^{(k)})_{-d},\F_p\bigr),
\end{equation*}
with differential dual to that of $\gr_F\cC^{(k)}$.
By Lemma~\ref{lem:homology-associated-graded} and equivariant
Poincar\'e--Lefschetz duality, for every $d$ there is an
$\SL_2(\F_p)$-equivariant isomorphism
\begin{equation*}
  H_d(\mathcal D^{(k)})
  \cong \gr_F H_d(B_k(\Sigma_{1,1}^\circ);\F_p),
\end{equation*}
where $F$ is the filtration on configuration-space homology induced by
the surface divided-power filtration $F_\bullet\cC^{(k)}$ under this
duality.

\begin{lemma}
  \label{lem:bs-ce-comparison}
  For every $k<2p$, there is an $\SL_2(\F_p)$-equivariant quasi-isomorphism
  \begin{equation}
    \mathcal D^{(k)}\simeq\mathcal E^{(k)}.
    \label{eq:bs-ce-quasi-isomorphism}
  \end{equation}
  In particular, for every $d$, there is an isomorphism of $\SL_2(\F_p)$-modules
  \begin{equation}
    \gr_F H_d(B_k(\Sigma_{1,1}^\circ);\F_p)
    \cong H_d\bigl(\mathcal E^{(k)}\bigr).
    \label{eq:low-weight-bs-ce-comparison}
  \end{equation}
\end{lemma}

\begin{proof}
  We first identify the dual cellular model with an associated-graded cobar
  complex and split it into three factors.  We then compare the factors with
  the three factors of $\mathcal E$; the bound $k<2p$ is used precisely in the
  polynomial--divided-power comparison for the middle factor.

  For an augmented algebra $R$ and a left $R$-module $U$, the reduced
  two-sided bar complex and the corresponding free $R$-resolution are
  \begin{equation*}
    \bar{B}_\bullet(\F_p,R,U)
    =\F_p\otimes\bar{R}^{\otimes\bullet}\otimes U,
    \qquad
    B_\bullet(R,R,U)
    =R\otimes\bar{R}^{\otimes\bullet}\otimes U.
  \end{equation*}
  The first computes $\F_p\otimes_R^{\mathbf L}U$; since each weight is
  finite-dimensional, its weightwise dual is the cobar cochain complex
  $\operatorname{Hom}_R(B_\bullet(R,R,U),\F_p)$ computing
  $\operatorname{Ext}_R(U,\F_p)$.

  Applying this construction to~\eqref{eq:bs-bar-complex},
  Bianchi--Stavrou identify the weightwise dual of $\cC$ with the cobar
  complex used in their
  Theorem~B~\cite[Propositions~6.1 and~6.3]{bianchiHomologyConfigurationSpaces2024}.
  Since the surface divided-power filtration is a direct sum by exact degree,
  as in the proof of Lemma~\ref{lem:homology-associated-graded}, the same
  description identifies $\mathcal D^{(k)}$ with the weight-$k$ summand of the
  corresponding associated-graded cobar complex.
  The Eilenberg--Zilber map and the K\"unneth decomposition in the proof of
  that theorem give a quasi-isomorphism from this cobar complex to the tensor
  product of the three cobar complexes computing
  \begin{equation}
    \operatorname{Ext}_{\Lambda(x)}(\F_p,\F_p)
    \otimes
    \operatorname{Ext}_{\Gamma(y)}(M_1,\F_p)
    \otimes
    \operatorname{Hom}_{\F_p}
    \bigl(\Gamma(y_a,y_b),\F_p\bigr),
    \label{eq:bs-ext-kunneth}
  \end{equation}
  where
  \begin{equation*}
    M_1=\Lambda(x_a,x_b),
    \qquad
    y\cdot1=2x_ax_b,
  \end{equation*}
  and $y$ annihilates $x_a,x_b,x_ax_b$~\cite[Definition~6.7 and the proof of
  Theorem~B]{bianchiHomologyConfigurationSpaces2024}.
  After passage to the
  associated graded, $\SL_2(\F_p)$ acts linearly on both pairs
  $(x_a,x_b)$ and $(y_a,y_b)$ by
  Lemma~\ref{lem:filtered-action}.
  It fixes the boundary variables $x,y$, preserves
  $M_1=\Lambda(x_a,x_b)$ and $\Gamma(y_a,y_b)$ separately, and the relation
  $y\cdot1=2x_ax_b$ is equivariant because $x_ax_b$ spans the trivial
  determinant representation.  Thus it preserves the three-factor
  decomposition~\eqref{eq:bs-ext-kunneth}.  The bar and cobar constructions,
  the Eilenberg--Zilber map, and the K\"unneth comparison are natural for these
  algebra and module automorphisms, so this decomposition is
  $\SL_2(\F_p)$-equivariant.

  We compare the three cobar factors underlying~\eqref{eq:bs-ext-kunneth}
  with $\mathcal E$.
  First, the normalized bars
  \begin{equation*}
    [x|\cdots|x]^\vee,
    \qquad r\text{ copies of }x,
  \end{equation*}
  have zero differential because $x^2=0$.
  Sending each such bar to
  $E_0^{[r]}$ identifies the first factor, as a bigraded complex, with
  $\Gamma(E_0)$.

  Second, the weightwise dual of $\Gamma(y_a,y_b)$ gives the third factor
  in~\eqref{eq:bs-ext-kunneth}.
  In weight $k<2p$, its homogeneous degree in
  $y_a,y_b$ is at most $p-1$.
  All relevant factorials are therefore
  invertible in $\F_p$, and the usual symmetrization map, together with the
  symplectic identification $H^\vee\cong H$, gives an
  $\SL_2(\F_p)$-equivariant identification of this factor with
  \begin{equation*}
    \Gamma(F_a,F_b).
  \end{equation*}

  It remains to examine the middle factor.
  Let $\mathcal P=\F_p[t]$, with $t$ of weight $-2$, and regard $M_1$ as a
  $\mathcal P$-module by
  \begin{equation*}
    t\cdot1=2x_ax_b,
    \qquad
    t\cdot x_a=t\cdot x_b=t\cdot x_ax_b=0.
  \end{equation*}
  The homomorphism $\mathcal P\to\Gamma(y)$ sending $t$ to $y=y^{[1]}$ induces a map
  \begin{equation}
    \bar{B}_\bullet(\F_p,\mathcal P,M_1)
    \longrightarrow
    \bar{B}_\bullet(\F_p,\Gamma(y),M_1).
    \label{eq:polynomial-divided-power-bar-map}
  \end{equation}
  Give both algebras the trivial $\SL_2(\F_p)$-action.  The action on $M_1$
  commutes with their module actions, so this bar map and the standard Koszul
  replacement used below are $\SL_2(\F_p)$-equivariant.
  This map is an isomorphism in every total weight $-k$ with $k<2p$.
  Indeed, every bar tensor in such a weight involves divided powers whose
  exponents have sum less than $p$.  For $m<p$, the image of $t^m$ is
  \begin{equation*}
    y^m=m!\,y^{[m]},
  \end{equation*}
  and $m!$ is a unit in $\F_p$.  Moreover, whenever $m+m'<p$,
  \begin{equation*}
    \bigl(m!\,y^{[m]}\bigr)\bigl(m'!\,y^{[m']}\bigr)
    =(m+m')!\,y^{[m+m']},
  \end{equation*}
  so the internal bar faces agree under~\eqref{eq:polynomial-divided-power-bar-map}.
  The final bar face agrees because the $\mathcal P$-action just displayed is the
  restriction of the $\Gamma(y)$-action on $M_1$.

  We may therefore replace the middle cobar factor in weight $k<2p$ by the
  dual of the polynomial bar complex.
  The standard two-term Koszul resolution of the augmentation module over
  $\mathcal P$ gives the quasi-isomorphic cochain complex
  \begin{equation*}
    M_1^\vee\otimes\Lambda(\kappa_0)
    \cong
    \Lambda(O_a,O_b)\otimes\Lambda(\kappa_0).
  \end{equation*}
  Here $\kappa_0$ is the Koszul generator of weight $2$ and cochain degree
  $1$.
  The second identification again uses the symplectic self-duality
  $H^\vee\cong H$.
  The module relation $t\cdot1=2x_ax_b$ dualizes to
  \begin{equation*}
    d(O_aO_b)=2\kappa_0.
  \end{equation*}
  Since $p$ is odd, we may set $O_2=2\kappa_0$; the differential then becomes
  $d(O_aO_b)=O_2$, as in~\eqref{eq:punctured-ce-differential}.

  The strict weight bound is visible in the decomposition
  \begin{equation*}
    \Gamma(y)
    \cong
    \bigotimes_{i\geq0}\F_p[y_i]/(y_i^p),
    \qquad y_i=y^{[p^i]}.
  \end{equation*}
  The Koszul generator $\kappa_0$ for $y_0=y$ is the generator identified
  with $O_2$ above.  At weight $2p$, the relation $y_0^p=0$ contributes the
  degree-two cobar class $\beta_0$, while the new algebra generator
  $y_1=y^{[p]}$ contributes the degree-one cobar class $\alpha_1$.
  These are precisely the first features of the divided-power cobar complex
  absent from the polynomial Koszul model, which is why the comparison is
  restricted to $k<2p$.

  Tensoring the three comparisons gives an $\SL_2(\F_p)$-equivariant
  quasi-isomorphism with
  \begin{equation*}
    \Gamma(E_0,F_a,F_b)\otimes\Lambda(O_a,O_b,O_2)
  \end{equation*}
  and its differential is precisely the Chevalley--Eilenberg differential
  of~\eqref{eq:punctured-ce-differential}.
  Combining this quasi-isomorphism
  with~\eqref{eq:homology-associated-graded} and the duality regrading gives~\eqref{eq:low-weight-bs-ce-comparison}.
\end{proof}

\subsection{The associated-graded source calculation}
\label{sec:associated-graded-source}

We now compute the weight-$(2p-1)$, degree-$(2p-3)$ part of
$H_*(\mathcal E)$ and identify its $\SL_2(\F_p)$-module structure.

\begin{lemma}
  \label{lem:punctured-exterior-homology}
  The exterior part of the complex~\eqref{eq:punctured-ce-algebra} has homology
  \begin{equation}
    H\bigl(\Lambda(O_a,O_b,O_2),d\bigr)
    =\operatorname{span}
    \{1,O_a,O_b,O_aO_2,O_bO_2,O_aO_bO_2\}.
    \label{eq:punctured-exterior-homology}
  \end{equation}
\end{lemma}

\begin{proof}
  The exterior algebra has basis
  \begin{equation*}
    1,\ O_a,\ O_b,\ O_2,\ O_aO_b,\ O_aO_2,\ O_bO_2,\ O_aO_bO_2.
  \end{equation*}
  By~\eqref{eq:punctured-ce-differential}, its only non-zero differential is
  $d(O_aO_b)=O_2$.
  Hence $O_aO_b$ is not a cycle and $O_2$ is a boundary.
  The other six basis elements are cycles; in particular,
  \begin{equation*}
    d(O_aO_bO_2)=O_2^2=0,
  \end{equation*}
  because $O_2$ is exterior.
  The image of the differential is the line
  spanned by $O_2$, so none of these six cycles is a boundary.
\end{proof}

\begin{lemma}
  \label{lem:source-bidegree-basis}
  In weight $2p-1$ and degree $2p-3$, the homology of $\mathcal E$ has a basis
  consisting of the following two families:
  \begin{align}
    E_0^{[2]}O_\varepsilon
    F_a^{[i]}F_b^{[p-2-i]},
    &\qquad
    O_\varepsilon\in\{O_a,O_b\},
    \quad 0\leq i\leq p-2,
    \label{eq:first-source-family}\\
    E_0O_aO_bO_2
    F_a^{[i]}F_b^{[p-3-i]},
    &\qquad
    0\leq i\leq p-3.
    \label{eq:second-source-family}
  \end{align}
  Their dimensions are respectively $2(p-1)$ and $p-2$.
\end{lemma}

\begin{proof}
  For a homogeneous generator $z$, set the \emph{defect}:
  \begin{equation*}
    \operatorname{def}(z)=\operatorname{weight}(z)-\deg(z).
  \end{equation*}
  Table~\ref{tab:punctured-ce-generators} gives
  \begin{equation*}
    \operatorname{def}(E_0)=\operatorname{def}(O_2)=1,
    \qquad
    \operatorname{def}(O_a)=\operatorname{def}(O_b)=\operatorname{def}(F_a)=\operatorname{def}(F_b)=0.
  \end{equation*}
  Since the required weight and degree differ by $2$, a contributing monomial
  must have total defect $2$.
  By
  Lemma~\ref{lem:punctured-exterior-homology}, this means that the exponent of
  $E_0$ plus the exponent of $O_2$ is $2$.
  As $O_2$ is exterior, the only
  possibilities for these two exponents are $(2,0)$ and $(1,1)$.

  Suppose first that the exterior class contains no $O_2$.
  The possible
  exterior classes are $1,O_a,O_b$, and the defect condition forces the factor
  $E_0^{[2]}$.
  For the class $1$, the weight remaining for the $F$-variables is
  $2p-3$, which is odd and therefore impossible because both $F_a$ and $F_b$
  have weight $2$.
  For $O_a$ or $O_b$, the remaining weight is
  \begin{equation*}
    (2p-1)-2-1=2(p-2),
  \end{equation*}
  which gives exactly the monomials in~\eqref{eq:first-source-family}.

  Now suppose that the exterior class contains $O_2$.
  The possible homology
  classes are
  \begin{equation*}
    O_aO_2,\qquad O_bO_2,\qquad O_aO_bO_2.
  \end{equation*}
  The defect condition forces one factor of $E_0$.
  For $O_aO_2$ and $O_bO_2$,
  the remaining weight is $2p-5$, again odd.
  For $O_aO_bO_2$, it is
  \begin{equation*}
    (2p-1)-1-1-1-2=2(p-3),
  \end{equation*}
  giving exactly~\eqref{eq:second-source-family}.

  The first family has two choices of $O_\varepsilon$ and $p-1$
  choices of $i$, whereas the second has $p-2$ choices of $i$.
\end{proof}

\begin{proposition}
  \label{prop:source-calculation}
  In weight $2p-1$ and degree $2p-3$,
  \begin{equation}
    \gr_F M
    \cong
    \bigl(H\otimes\Gamma^{p-2}H\bigr)\oplus\Gamma^{p-3}H,
    \label{eq:source-associated-graded}
  \end{equation}
  as an $\SL_2(\F_p)$-module.
  Hence the $\widetilde G$-composition factors of $M$ are the inflations of the composition factors in~\eqref{eq:source-associated-graded}.
\end{proposition}

As a dimension check, in agreement with Bianchi--Stavrou's fake-basechange formula, we have:
\begin{equation*}
  \dim(\gr_F M)
  =2\dim\Gamma^{p-2}H+\dim\Gamma^{p-3}H
  =2(p-1)+(p-2)=3p-4.
\end{equation*}

\begin{proof}
  Since $2p-1<2p$, Lemma~\ref{lem:bs-ce-comparison} identifies $\gr_F M$ with
  the bidegree computed in
  Lemma~\ref{lem:source-bidegree-basis}.

  In the first family, $E_0^{[2]}$ is invariant,
  $\operatorname{span}\{O_a,O_b\}\cong H$, and the monomials
  \begin{equation*}
    F_a^{[i]}F_b^{[p-2-i]},
    \qquad 0\leq i\leq p-2,
  \end{equation*}
  form the standard basis of $\Gamma^{p-2}H$.
  Hence this family is
  $H\otimes\Gamma^{p-2}H$.

  In the second family, $E_0$ and $O_2$ are invariant, while $O_aO_b$ spans
  $\Lambda^2H=\det(H)$ and is therefore invariant under $\SL_2(\F_p)$.
  The remaining $F$-monomials form
  $\Gamma^{p-3}H$.
  Thus the second family is $\Gamma^{p-3}H$, proving~\eqref{eq:source-associated-graded}.

  The final assertion about $\widetilde G$-composition factors follows from
  Corollary~\ref{cor:filtered-composition-factors} and the definition of inflation in
  Section~\ref{sec:rep-conventions}.
\end{proof}

\section{Proof of Theorem~\ref{thmB}: torsion on the critical line}

We now prove Theorem~\ref{thmB} by showing that the punctured-torus torsion class $\tau_p$ survives to a non-zero class in $H_{2p-2}(B_{2p}(\Sigma_1);\Z)$ for $p \geq 5$.
Proposition~\ref{prop:survival} proves the assertion for $p\geq5$ using representation theory, and Proposition~\ref{prop:prime-three} proves the assertion for $p=3$ using Napolitano's calculation.

\subsection{No trivial source factors and survival for \texorpdfstring{$p\geq5$}{p >= 5}}
\label{sec:survival}

We combine the source calculation with the modular Clebsch--Gordan formula to exclude trivial composition factors.
Equivariance of the Gysin map then prevents its image from containing the fixed target class $\beta_0$.

\begin{lemma}
  \label{lem:no-trivial-source-factor}
  If $p\geq5$, then
  $\bigl(H\otimes\Gamma^{p-2}H\bigr)\oplus\Gamma^{p-3}H$
  has composition factors
  $V_{p-1}$, $V_{p-3}$, $V_{p-3}$,
  and in particular has no trivial composition factor.
\end{lemma}

\begin{proof}
  Applying~\eqref{eq:clebsch-gordan} with $r=p-2$ gives
  \[
    [H\otimes\Gamma^{p-2}H]
    =[V_{p-1}]+[V_{p-3}]
  \]
  in the Grothendieck group.
  Hence
  \[
    \bigl[
      (H\otimes\Gamma^{p-2}H)\oplus\Gamma^{p-3}H
    \bigr]
    =[V_{p-1}]+2[V_{p-3}].
  \]
  Since $p\geq5$, both $p-1$ and $p-3$ are positive; thus neither simple module appearing here is $V_0$, the trivial module.
\end{proof}

\begin{proposition}
  \label{prop:survival}
  If $p\geq5$, then $\beta_0\notin\operatorname{im}\bar\partial$.
  Consequently $j_*(\tau_p)$ is a non-zero order-$p$ class in
  $H_{2p-2}(B_{2p}(\Sigma_1);\Z_{(p)})$, and therefore
  $H_{2p-2}(B_{2p}(\Sigma_1);\Z)$ has $p$-torsion.
\end{proposition}

\begin{proof}
  Proposition~\ref{prop:source-calculation} and Lemma~\ref{lem:no-trivial-source-factor} show that $M$ has no trivial $\widetilde G$-composition factor.
  Since $\bar\partial:M\to N$ is $\widetilde G$-linear, the same is true of its image, which is a quotient of $M$.
  On the other hand, Lemma~\ref{lem:target-fixed} says that $\beta_0$ is a non-zero fixed vector of $N$.
  Therefore $\beta_0\notin\operatorname{im}\bar\partial$.

  Lemma~\ref{lem:mod-p-reduction} now gives $j_*(\tau_p)\neq0$.
  Since $\tau_p$ has order $p$, its non-zero image also has order $p$, so
  $\Tors H_{2p-2}(B_{2p}(\Sigma_1);\Z_{(p)})\neq0$.
  By~\eqref{eq:localization-torsion}, $H_{2p-2}(B_{2p}(\Sigma_1);\Z)$ has
  $p$-torsion.
\end{proof}

\subsection{The exceptional prime \texorpdfstring{$p=3$}{p = 3}}
\label{sec:prime-three}

The source--target separation used for $p\geq5$ fails at $p=3$.
We therefore use Napolitano's integral computation (Proposition~\ref{prop:napolitano}) to determine whether the punctured torsion class survives filling.

At $p=3$, the associated-graded source has composition factors
\begin{equation*}
  [H\otimes\Gamma^1H\oplus\Gamma^0H]
  =
  [V_2]+2 [V_0],
\end{equation*}
so the source contains two trivial composition factors.
Although $\beta_0$ remains a fixed vector by Lemma~\ref{lem:target-fixed}, equivariance no longer prevents it from lying in $\operatorname{im}\bar\partial$: this image is a quotient of the source and may inherit one of its trivial composition factors.
Using a result of Napolitano~\cite{napolitanoCohomologyConfigurationSpaces2003}, we now show that the torsion class $\tau_3$ is indeed in the image of the connecting homomorphism $\partial$.

\begin{proposition}
  \label{prop:prime-three}
  The group $H_4(B_6(\Sigma_1);\Z)$ has no $3$-torsion, hence $j_*(\tau_3)=0$ and, equivalently, $\tau_3\in\operatorname{im}\partial$.
\end{proposition}

\begin{proof}
  For $B_6(\Sigma_1)$, Napolitano's calculation (Proposition~\ref{prop:napolitano}) gives
  $H^5(B_6(\Sigma_1);\Z)\cong \Z^9\oplus(\Z/2)^8$.
  Hence $H^5(B_6(\Sigma_1);\Z)$ has no $3$-torsion. By the universal coefficient theorem, there is a short exact sequence:
  \begin{equation*}
    0\to
    \Ext^1_{\Z}(H_4(B_6(\Sigma_1);\Z),\Z)
    \to
    H^5(B_6(\Sigma_1);\Z)
    \to
    \Hom_{\Z}(H_5(B_6(\Sigma_1);\Z),\Z)
    \to0.
  \end{equation*}
  The right-hand term is free, and finite generation together with the structure
  theorem identifies the left-hand term with the torsion subgroup of
  $H_4(B_6(\Sigma_1);\Z)$.
  Therefore $\Tors_3H_4(B_6(\Sigma_1);\Z)=0$.

  On the other hand, Theorem~\ref{thm:punctured-threshold} gives $\Tors_3 H_4(B_6(\Sigma_{1,1}^\circ);\Z)=\Z/3\{\tau_3\}$.
  The image $j_*(\tau_3)\in H_4(B_6(\Sigma_1);\Z)$ has order dividing $3$.
  Since $H_4(B_6(\Sigma_1);\Z)$ has no $3$-torsion, this image must be zero.
  By exactness in~\eqref{eq:gysin-survival}, this is equivalent to
  $\tau_3\in\operatorname{im}\partial$.
\end{proof}

\section{Proof of Theorem~\ref{thmC}: no torsion in high degree}
\label{sec:high-degree}

We now prove Theorem~\ref{thmC} by analyzing the Borel--Moore chain complex of Napolitano's cellular decomposition~\cite{napolitanoCohomologyConfigurationSpaces2003}.
After translating ordinary high-degree torsion into bottom-degree Borel--Moore torsion, we identify the relevant punctured cells and compute the closed--punctured connecting map.

Note that Napolitano computes Borel--Moore homology, not ordinary homology.
The unordered configuration space $B_k(\Sigma_1)$ is an orientable real
$2k$-manifold. Therefore, Borel--Moore Poincaré duality gives
\begin{equation*}
  \HBM_i(B_k(\Sigma_1);\Z_{(p)}) \cong H^{2k-i}(B_k(\Sigma_1);\Z_{(p)}).
\end{equation*}
The universal coefficient theorem over $\Z_{(p)}$ gives
\begin{equation*}
  0\to
  \Ext^1_{\Z_{(p)}}(H_{q-1}(B_k(\Sigma_1);\Z_{(p)}),\Z_{(p)})
  \to
  H^q(B_k(\Sigma_1);\Z_{(p)})
  \to
  \Hom_{\Z_{(p)}}(H_q(B_k(\Sigma_1);\Z_{(p)}),\Z_{(p)})
  \to0.
\end{equation*}
By finite generation, the structure theorem over $\Z_{(p)}$ identifies the
$\Ext$-term with the $p$-primary torsion in
$H_{q-1}(B_k(\Sigma_1);\Z)$, while the $\Hom$-term is a finite free
$\Z_{(p)}$-module.  The short exact sequence therefore splits as a sequence of
$\Z_{(p)}$-modules.
Consequently, we have the following.

\begin{lemma}
  \label{lem:bm-uct}
  For every $n$ and $r$, the $p$-primary torsion of $H_r(B_n(\Sigma_1);\Z)$ is recorded in the $p$-primary torsion of $\HBM_{2n-r-1}(B_n(\Sigma_1);\Z_{(p)})$.
  In particular, torsion-freeness of the latter group implies that $H_r(B_n(\Sigma_1);\Z)$ has no $p$-torsion.
\end{lemma}

\subsection{The cell model and the closed--punctured exact sequence}
\label{sec:cell-model}

We fix orientations and write the cellular differential in the form needed for the torsion calculation.
Separating cells according to whether the base north pole is occupied then yields a short exact sequence relating the closed and punctured complexes.

Recall from Section~\ref{sec:napolitano} that Napolitano's cells are indexed by tuples of the form:
\begin{equation}\label{eq:napolitano-tuple}
  (e,v;(m_1,v_1),\ldots,(m_d,v_d)),
\end{equation}
with weight and Borel--Moore degree:
\begin{equation}
  k = e+v+\sum_i(m_i+v_i),
  \qquad
  n = e+\sum_i(m_i+1).
  \label{eq:napolitano-indexing}
\end{equation}
If we view $\Sigma_1 = S^1 \cup (S^1 \times \R)$, $v$ counts whether the north pole of $S^1$ is occupied, $e$ counts the number of points on the base circle (but not at the north pole), and $(m_i,v_i)$ counts the number of points in the $i$th fiber of the cylinder and whether that fiber contains the north pole.
We will say that the $i$th fiber is \emph{pinned} if $v_i=1$.
A point not at the north pole is called \emph{off-pole}.

We fix the cell orientation used below and record all five signed boundary
formulas.  The general signs will not be rederived here; the exact signs needed
in the proof are derived after restriction to pure cells in
Lemma~\ref{lem:connecting-pure}.
Orient a cell by listing first the $e$ ordered arc coordinates and then, for each fiber from left to right, its position followed by its $m_i$ ordered off-pole coordinates.
Let us denote
\begin{equation*}
  P(a,b)=
  \begin{cases}
    0,&a,b\text{ both odd},\\
    \displaystyle\binom{\lfloor(a+b)/2\rfloor}{\lfloor a/2\rfloor},
    &\text{otherwise}.
  \end{cases}
\end{equation*}
Consider the cell indexed by~\eqref{eq:napolitano-tuple}.
It has five types of codimension-one boundary cells (if the corresponding condition is not satisfied, then the face does not exist):
\begin{enumerate}
  \item Two adjacent fibers $i$, $i+1$ can merge unless both are pinned.
    The indexing tuple of the boundary face replaces $(m_i,v_i),(m_{i+1},v_{i+1})$ with $(m_i+m_{i+1},\max(v_i,v_{i+1}))$, with coefficient $(-1)^{e+i-1+\sum_{j\le i}m_j}P(m_i,m_{i+1})$.
  \item If $v_i=0$ and $m_i\ge1$, an off-pole point can reach the north pole.
    The boundary face replaces $(m_i,0)$ by $(m_i-1,1)$, with coefficient $(-1)^{e+i-1+\sum_{j<i}m_j}\bigl(1+(-1)^{m_i}\bigr)$.
  \item If $v=0$ and $e\ge1$, a point on the base circle can reach the base north pole.
    The boundary face replaces $(e,0)$ by $(e-1,1)$, leaving the fibers unchanged, with coefficient $-\bigl(1+(-1)^e\bigr)$.
  \item If the base north pole and the leftmost north pole are not both occupied, the
    leftmost fiber can slide onto the base circle.
    Its $m_1$ off-pole points join the base circle, its potential pin is transferred to the base north pole, and the coefficient is $-(-1)^eP(e,m_1)$.
  \item The analogous rightmost slide has coefficient $(-1)^{e+S(1+m_d)}P(e,m_d)$, where $S=(d-1)+\sum_{j<d}m_j$.
\end{enumerate}
All other codimension-one degenerations either collide points or escape an
interior fiber and therefore land at the compactification basepoint.  The
incidence types and magnitudes come from Napolitano's
construction~\cite[Sections~2.3--2.4]{napolitanoCohomologyConfigurationSpaces2003};
the displayed signs are those obtained from the orientation convention above.
For the pure cells used below, types (1)--(2) are absent, and the exact signs of
types (3)--(5) are derived directly in the proof of
Lemma~\ref{lem:connecting-pure}.

The full (``closed'') complex $C^{\mathrm{cl}}(k)$ is the Borel--Moore chain complex of the cellular decomposition of $B_k(\Sigma_1)$.
The punctured complex $C^{\mathrm{pu}}(k)$ is the quotient obtained by taking $v=0$ and discarding
summands that hit $v=1$.
Both are graded by the Borel--Moore degree $n$ of~\eqref{eq:napolitano-indexing}, and we write $H_i(-)$ for the homology of a chain complex, so that
\begin{equation*}
  H_i(C^{\mathrm{cl}}(k))\cong\HBM_i(B_k(\Sigma_1)),
  \qquad
  H_i(C^{\mathrm{pu}}(k))\cong\HBM_i\bigl(B_k(\Sigma_1\setminus\{\ast\})\bigr),
\end{equation*}
where $\ast$ denotes the base north pole.
For punctured cells, $n-k=d-\sum_i v_i\ge0$, so the punctured complex for $B_k(\Sigma_{1,1}^\circ)$ has no cells in Borel--Moore degree below $k$.
For closed cells, $n-k=d-\sum_i v_i-v\ge -1$, so the closed complex has no cells in Borel--Moore degree below $k-1$.

The subcomplex of closed cells with $v=1$ is naturally identified
with the punctured complex in weight $k-1$.
Thus there is a short exact sequence of Borel--Moore chain complexes
\begin{equation}
  0\longrightarrow C^{\mathrm{pu}}(k-1) \longrightarrow C^{\mathrm{cl}}(k)
  \longrightarrow C^{\mathrm{pu}}(k) \longrightarrow 0.
  \label{eq:closed-punctured-complexes}
\end{equation}

\subsection{Bottom-degree punctured cells}
\label{sec:all-pinned}

We describe the bottom Borel--Moore homology of the punctured complex in terms of all-pinned cells.
We show that pure all-pinned cells generate in weights at most $2p$ and form a basis in weight $2p-1$.

Recall that there are no cells in Borel--Moore degree $< k$ in the punctured complex.
The bottom-degree punctured cells are the \emph{all-pinned cells}:
\begin{equation}
  X(e;m_1,\dots,m_d)
  =(e,0;(m_1,1),\ldots,(m_d,1)),
  \qquad
  e+d+\sum_i m_i=k.
  \label{eq:x-cell}
\end{equation}
The differential vanishes on these cells: merging two pinned fibers is forbidden, and every other codimension-one face (sliding a pinned fiber onto the base circle, or moving a base point to the base north pole) occupies the base north pole and is therefore discarded in the punctured complex.

Let
\begin{equation*}
  W_k \coloneqq H_k(C^{\mathrm{pu}}(k)) \; (\cong H^k(B_k(\Sigma_{1,1}^\circ))).
\end{equation*}
In Borel--Moore degree $k+1$, the punctured complex consists of cells with exactly one unpinned fiber.
Thus $W_k = C_k^{\mathrm{pu}}(k) / d(C_{k+1}^{\mathrm{pu}}(k))$ has three types of relations that we now describe.
Suppose that
\begin{equation}
  Y(e; m_1, \dots, m_d; j) \coloneqq (e,0; (m_1,1), \dots, (m_{j-1},1), (m_j,0), (m_{j+1},1), \dots, (m_d,1))
  \label{eq:y-cell}
\end{equation}
is a \emph{one-unpinned} punctured cell.
Necessarily $m_j\geq1$, since the indexing condition $m_j+v_j\geq1$ and the
unpinned condition $v_j=0$ hold.
Then the boundary of $Y(e; m_1, \dots, m_d; j)$ has three possible types of terms (we will not need the precise signs here):
\begin{enumerate}
  \item merging two adjacent fibers, if possible, with coefficients $\pm P(m_{j-1}, m_j)$ or $\pm P(m_j, m_{j+1})$;
  \item sliding an off-pole point to the north pole, if $m_j\ge1$, with coefficient $\pm(1+(-1)^{m_j})$;
  \item sliding the unpinned fiber onto the base circle, if the unpinned fiber is the leftmost or rightmost fiber, with coefficient $\pm P(e, m_j)$.
\end{enumerate}

Let us write $\Pi(e,d) \coloneqq X(e;0^d)$ ($e+d=k$) for the ``pure'' all-pinned cell with $d$ fibers and $e$ points on the base circle.

\begin{lemma}
  \label{lem:pure-generate}
  If $k\leq2p$, then the classes $\Pi(e,d)$ with $e+d=k$ generate $W_k \otimes \Z_{(p)}$.
\end{lemma}

\begin{proof}
  The strategy is to filter by the total off-pole mass and kill every positive
  associated-graded piece with a unit merge coefficient, inducting on the
  position of the first non-zero block.
  Let us set $s(X) = \sum_i m_i$, where $X$ is a Borel--Moore cell.
  Note that pure cells have $s(\Pi(e,d))=0$.
  This induces a filtration on $W_k$ and we let $\gr_s W_k$ be the associated graded.

  If $Y$ is a cell of type~\eqref{eq:y-cell}, then $dY$ is a linear combination of cells of type~\eqref{eq:x-cell}, each with $s$-value at most $s(Y)$: merging two fibers does not change $s$, sliding an off-pole point to the north pole decreases $s$ by $1$, and sliding the unpinned fiber onto the base circle decreases $s$ by $m_j > 0$.
  Therefore, in $\gr_s W_k$, the remaining relations are:
  \begin{equation}
    \pm P(m_{j-1}, m_j) X(e; m_1, \dots, m_{j-1}+m_j, \dots, m_d) \pm P(m_j, m_{j+1}) X(e; m_1, \dots, m_j+m_{j+1}, \dots, m_d) = 0.
    \label{eq:gr-s-relations}
  \end{equation}
  Note that if the unpinned fiber is at the beginning or the end, then there is only one term in~\eqref{eq:gr-s-relations}.

  Let us now prove that $\gr_s W_k \otimes \Z_{(p)} = 0$ for $s \geq 1$, from which it follows that every class in $W_k \otimes \Z_{(p)}$ is a linear combination of pure cells.
  We proceed by induction on the position of the first non-zero $m_i$ in the tuple $(m_1, \dots, m_d)$.

  Since $e+d+\sum_i m_i=k\leq2p$ and $d\geq1$, every $m_i$ is at most
  $2p-1$.  Thus, when $m_i$ is positive and even,
  $1\leq m_i/2\leq p-1$ and $m_i/2$ is a unit in $\Z_{(p)}$.

  First, consider an all-pinned cell $X(e; m_1, \dots, m_d)$ with $m_1 > 0$.
  We can find $b,c$ such that $b+c=m_1$ and $P(b,c)$ is a unit in $\Z_{(p)}$: if $m_1$ is odd, let $(b,c) = (1, m_1-1)$ (then $P(b,c)=1$), and if $m_1$ is even, let $(b,c) = (2, m_1-2)$ (then $P(b,c)=m_1/2$, which is a unit in $\Z_{(p)}$ because $m_1 < 2p$).
  Then the relation coming from $Y = Y(e; b,c, m_2, \dots, m_d; 1)$ is just $\pm P(b,c) X(e; m_1, \dots, m_d) = 0$, which implies that $X(e; m_1, \dots, m_d) = 0$ in $\gr_s W_k \otimes \Z_{(p)}$.

  Now suppose that the result has been proved for all cells with first non-zero $m_i$ at position $< j$, and consider a cell $X(e; 0, \dots, 0, m_j, \dots, m_d)$ with $m_j > 0$ and $m_i = 0$ for all $i < j$.
  Choose $m_j = b+c$ with $P(b,c)$ a unit in $\Z_{(p)}$ as above, and consider the relation coming from $Y = Y(e; 0, \dots, 0, b, c, m_{j+1}, \dots, m_d; j)$, whose unpinned fiber, in position $j$, has left neighbor $(0,1)$ and right neighbor $(c,1)$.
  The relation is
  \begin{equation*}
    \pm P(0,b) X(e; 0, \dots, 0, b, c, m_{j+1}, \dots, m_d) \pm P(b,c) X(e; 0, \dots, 0, m_j, m_{j+1}, \dots, m_d) = 0,
  \end{equation*}
  where the first cell has $j-2$ zeros and the second one has $j-1$ zeros.
  The first cell has its first non-zero entry at position $j-1$, so it vanishes in $\gr_s W_k \otimes \Z_{(p)}$ by the induction hypothesis.
  Since $P(b,c)$ is a unit in $\Z_{(p)}$, we conclude that $X(e; 0, \dots, 0, m_j, \dots, m_d) = 0$ in $\gr_s W_k \otimes \Z_{(p)}$.
\end{proof}

\begin{lemma}
  \label{lem:pure-basis}
  The $2p$ classes $\Pi(e,d)$ with $e+d=2p-1$ form a $\Z_{(p)}$-basis of $W_{2p-1}\otimes\Z_{(p)}$.
\end{lemma}

\begin{proof}
  By Lemma~\ref{lem:pure-generate}, there is a surjection from $\Z_{(p)}\{\Pi(e,d) \mid e+d=2p-1\}$ to $W_{2p-1}\otimes\Z_{(p)}$.
  By~\cite[Proposition~3.5]{drummond-coleBettiNumbersConfiguration2017}, the $\Q$-dimension of $W_{2p-1}\otimes\Q$ is $2p$, so the kernel has rank $0$.
  A rank $0$ submodule of a free module over a PID is trivial, so the surjection is an isomorphism.
\end{proof}

\subsection{The connecting map on pure cells}
\label{sec:connecting-map}

We compute the connecting map on pure generators; its cokernel will give the bottom closed-torus group needed for Theorem~\ref{thmC}.
The map induced by~\eqref{eq:closed-punctured-complexes} is
\begin{equation*}
  \delta:W_k\otimes\Z_{(p)}\longrightarrow W_{k-1}\otimes\Z_{(p)}.
\end{equation*}
It sends a punctured cycle in weight $k$ to a punctured cycle in weight $k-1$ by lifting it to a closed chain, applying the closed differential, and identifying the result with a punctured cycle.

\begin{lemma}
  \label{lem:connecting-pure}
  If $e+d=k$, then
  \begin{equation}
    \delta\Pi(e,d)
    =-\bigl(1+(-1)^e\bigr)\mathbbm{1}_{e\geq1}\Pi(e-1,d)
    +(-1)^e\bigl((-1)^{d-1}-1\bigr)\mathbbm{1}_{d\geq1}\Pi(e,d-1).
    \label{eq:connecting-pure}
  \end{equation}
\end{lemma}

\begin{proof}
  Recall that the pure cell $\Pi(e,d)$ has $d$ pinned fibers and $e$ off-pole points on the base circle.
  It is a cycle in the punctured complex: merging adjacent fibers is impossible, and every other codimension-one face occupies the base north pole and is therefore discarded in the punctured complex.

  To compute $\delta\Pi(e,d)$, we lift $\Pi(e,d)$ to the closed complex and apply the closed differential.
  Write its orientation form as
  \begin{equation*}
    \omega_{e,d}=d\theta_1\wedge\cdots\wedge d\theta_e
    \wedge dt_1\wedge\cdots\wedge dt_d.
  \end{equation*}
  For the two fiber endpoints, the outward-normal-first convention gives
  \begin{equation*}
    \iota_{-\partial_{t_1}}\omega_{e,d}=-(-1)^e\omega_{e,d-1},
    \qquad
    \iota_{\partial_{t_d}}\omega_{e,d}=(-1)^{e+d-1}\omega_{e,d-1}.
  \end{equation*}
  When $d=1$, $t_1=0$ and $t_1=1$ are distinct boundary faces, although
  the torus gluing identifies their targets.  Similarly, the two base-circle
  endpoint faces have coefficients $-1$ and $(-1)^{e-1}$, whose sum is
  $-(1+(-1)^e)$.
  There are only three kinds of possibly non-zero contributions:
  \begin{itemize}
    \item If $e\ge1$, then the base north pole can be occupied by sliding an off-pole point from the base circle.
      The resulting cell is $\Pi(e-1,d)$, with coefficient $-(1+(-1)^e)$.
    \item If $d\ge1$, then the leftmost pinned fiber can slide onto the base circle.
      The resulting cell is $\Pi(e,d-1)$, with coefficient $-(-1)^eP(e,0)=-(-1)^e$.
    \item If $d\ge1$, then the rightmost pinned fiber can slide onto the base circle.
      The resulting cell is $\Pi(e,d-1)$, with coefficient $(-1)^{e+(d-1)}P(e,0)=(-1)^{e+d-1}$.
  \end{itemize}
  Since $-(-1)^e+(-1)^{e+d-1}=(-1)^e\bigl((-1)^{d-1}-1\bigr)$, the sum of these three contributions gives~\eqref{eq:connecting-pure}.
\end{proof}

\begin{proposition}
  \label{prop:cokernel-delta}
  For every odd prime $p$,
  \begin{equation*}
    \operatorname{coker}\bigl(
      \delta:W_{2p}\otimes\Z_{(p)}\longrightarrow W_{2p-1}\otimes\Z_{(p)}
    \bigr)
    \cong\Z_{(p)}^{p-1}.
  \end{equation*}
\end{proposition}

\begin{proof}
  By Lemma~\ref{lem:pure-generate}, the source $W_{2p} \otimes \Z_{(p)}$ is generated by the pure cells $\Pi(e,d)$ with $e+d=2p$, so the image of $\delta$ is spanned by their images.
  There are $p+1$ such cells with both $e$ and $d$ even, and $p$ such cells with both $e$ and $d$ odd.
  According to Lemma~\ref{lem:connecting-pure}, the connecting map kills $\Pi(e,d)$ when both $e$ and $d$ are odd, and satisfies $\delta\Pi(e,d) = \pm2\Pi(e-1,d) \pm2\Pi(e,d-1)$ when both $e$ and $d$ are even (if $e=0$ or $d=0$, the corresponding term is absent, as recorded by the indicator functions in~\eqref{eq:connecting-pure}).
  Note that $2$ is a unit in $\Z_{(p)}$.

  By Lemma~\ref{lem:pure-basis}, the target $W_{2p-1} \otimes \Z_{(p)}$ is freely spanned by pure cells $\Pi(e,d)$ with $e+d=2p-1$.
  The supports of the $p+1$ non-zero images $\delta\Pi(e,d)$ (with both $e$ and $d$ even) are pairwise disjoint.
  Indeed, a target of the form $\Pi(2i-1,2j)$ can only be hit by $\delta\Pi(2i,2j)$, and a target of the form $\Pi(2i,2j-1)$ can only be hit by $\delta\Pi(2i,2j)$.
  Each non-zero image has a unit coefficient on its support and is therefore a
  primitive vector in the free submodule spanned by that support.  Since the
  supports are pairwise disjoint, independent basis changes on these support
  blocks extend the $p+1$ images to part of a basis of
  $W_{2p-1}\otimes\Z_{(p)}$.  Thus $\operatorname{im}\delta$ is a free direct
  summand of rank $p+1$, and the cokernel is free of rank
  $(2p)-(p+1)=p-1$.
\end{proof}

\begin{proof}[Proof of Theorem~\ref{thmC}]
  Part of the long exact homology sequence induced by the short exact sequence of chain complexes~\eqref{eq:closed-punctured-complexes} (with $\Z_{(p)}$-coefficients) is:
  \begin{equation*}
    H_{2p}(C^{\mathrm{pu}}(2p))
    \xrightarrow{\delta}
    H_{2p-1}(C^{\mathrm{pu}}(2p-1))
    \longrightarrow
    H_{2p-1}(C^{\mathrm{cl}}(2p))
    \longrightarrow
    H_{2p-1}(C^{\mathrm{pu}}(2p)).
  \end{equation*}
  The last group vanishes as $C^{\mathrm{pu}}(2p)$ has no cells in Borel--Moore degree $< 2p$.
  Moreover, we have $H_{2p-1}(C^{\mathrm{pu}}(2p-1)) = W_{2p-1}\otimes\Z_{(p)}$ and $H_{2p}(C^{\mathrm{pu}}(2p)) = W_{2p}\otimes\Z_{(p)}$, so Proposition~\ref{prop:cokernel-delta} gives
  \begin{equation*}
    H_{2p-1}(C^{\mathrm{cl}}(2p)) \cong \operatorname{coker}(\delta) \cong \Z_{(p)}^{p-1}.
  \end{equation*}
  In particular, this is a free $\Z_{(p)}$-module.
  Since $H_{2p-1}(C^{\mathrm{cl}}(2p))\cong\HBM_{2p-1}(B_{2p}(\Sigma_1);\Z_{(p)})$, we conclude by Lemma~\ref{lem:bm-uct}, applied with $n=2p$ and $r=2p$, that $H_{2p}(B_{2p}(\Sigma_1);\Z)$ has no $p$-torsion.

  Similarly, $H_{2p-2}(C^{\mathrm{cl}}(2p))\cong\HBM_{2p-2}(B_{2p}(\Sigma_1);\Z_{(p)})$ vanishes, because the closed complex has no cells in Borel--Moore degree below $2p-1$.
  Lemma~\ref{lem:bm-uct}, applied with $n=2p$ and $r=2p+1$, then shows that $H_{2p+1}(B_{2p}(\Sigma_1);\Z)$ has no $p$-torsion.
\end{proof}

\printbibliography

@article{anAsymptoticCorrigendum2026,
  author     = {An, Byung Hee and Drummond-Cole, Gabriel C. and Knudsen, Ben},
  date       = {2026},
  eprint     = {2602.20228},
  eprinttype = {arXiv},
  title      = {Corrigendum to ``Asymptotic Homology of Graph Braid Groups''},
}

@article{anAsymptoticHomology2022,
  author       = {An, Byung Hee and Drummond-Cole, Gabriel C. and Knudsen, Ben},
  date         = {2022},
  doi          = {10.2140/gt.2022.26.1745},
  eprint       = {2005.08286},
  eprinttype   = {arXiv},
  journaltitle = {Geom. Topol.},
  pages        = {1745--1771},
  title        = {Asymptotic Homology of Graph Braid Groups},
  volume       = {26},
}

@article{bassSolutionCongruence1967,
  author       = {Bass, Hyman and Milnor, John and Serre, Jean-Pierre},
  date         = {1967},
  doi          = {10.1007/BF02684586},
  issn         = {0073-8301},
  journaltitle = {Publ. Math. Inst. Hautes Études Sci.},
  pages        = {59--137},
  title        = {Solution of the Congruence Subgroup Problem for {$SL_n$} ($n\geq3$) and {$Sp_{2n}$} ($n\geq2$)},
  volume       = {33},
}

@article{bianchiHomologyConfigurationSpaces2024,
  author       = {Bianchi, Andrea and Stavrou, Andreas},
  date         = {2024},
  doi          = {10.1515/crelle-2024-0029},
  eprint       = {2307.08664},
  eprinttype   = {arXiv},
  issn         = {0075-4102},
  journaltitle = {J. Reine Angew. Math.},
  pages        = {197--258},
  title        = {Homology of Configuration Spaces of Surfaces modulo an Odd Prime},
  volume       = {813},
}

@article{bodigheimerHomologyConfigurationSpaces1989,
  author       = {Bödigheimer, Carl-Friedrich and Cohen, Fred and Taylor, Laurence},
  date         = {1989},
  doi          = {10.1016/0040-9383(89)90035-9},
  issn         = {0040-9383},
  journaltitle = {Topology},
  mrnumber     = {991102},
  number       = {1},
  pages        = {111--123},
  title        = {On the Homology of Configuration Spaces},
  volume       = {28},
}

@article{chenModHomologyUnordered2024,
  author       = {Chen, Matthew and Zhang, Adela YiYu},
  date         = {2024},
  doi          = {10.1090/proc/16683},
  eprint       = {2208.10293},
  eprinttype   = {arXiv},
  issn         = {0002-9939},
  journaltitle = {Proc. Am. Math. Soc.},
  number       = {5},
  pages        = {2239--2248},
  title        = {Mod $p$ Homology of Unordered Configuration Spaces of $p$ Points in Parallelizable Surfaces},
  volume       = {152},
}

@article{chettihHomologyConfiguration2018,
  author       = {Chettih, Safia and Lütgehetmann, Daniel},
  date         = {2018},
  doi          = {10.2140/agt.2018.18.2443},
  eprint       = {1612.08290},
  eprinttype   = {arXiv},
  issn         = {1472-2747},
  journaltitle = {Algebr. Geom. Topol.},
  number       = {4},
  pages        = {2443--2469},
  title        = {The Homology of Configuration Spaces of Trees with Loops},
  volume       = {18},
}

@book{cohenHomologyIteratedLoop1976,
  author    = {Cohen, Frederick R. and Lada, Thomas J. and May, J. Peter},
  date      = {1976},
  doi       = {10.1007/BFb0080464},
  publisher = {Springer},
  series    = {Lecture Notes in Mathematics},
  title     = {The Homology of Iterated Loop Spaces},
  volume    = {533},
}

@article{cooperTwoModP2015,
  author       = {Cooper, James},
  date         = {2015},
  doi          = {10.1142/S1793525315500120},
  eprint       = {1402.4186},
  eprinttype   = {arXiv},
  journaltitle = {J. Topol. Anal.},
  number       = {2},
  pages        = {309--343},
  title        = {Two Mod-$p$ Johnson Filtrations},
  volume       = {7},
}

@article{DotyHenke2005,
  author  = {Doty, Stephen R. and Henke, Anne E.},
  journal = {Q. J. Math.},
  number  = {2},
  pages   = {189--207},
  title   = {Decomposition of tensor products of modular irreducibles for {$SL_2$}},
  volume  = {56},
  year    = {2005},
}

@article{drummond-coleBettiNumbersConfiguration2017,
  author       = {Drummond-Cole, Gabriel C. and Knudsen, Ben},
  date         = {2017},
  doi          = {10.1112/jlms.12066},
  eprint       = {1608.07490},
  eprinttype   = {arXiv},
  issn         = {0024-6107},
  journaltitle = {J. Lond. Math. Soc.},
  mrnumber     = {3708955},
  number       = {2},
  pages        = {367--393},
  series       = {2},
  title        = {Betti Numbers of Configuration Spaces of Surfaces},
  volume       = {96},
}

@article{fadellConfigurationSpaces1962,
  author       = {Fadell, Edward and Neuwirth, Lee},
  date         = {1962},
  doi          = {10.7146/math.scand.a-10517},
  journaltitle = {Math. Scand.},
  pages        = {111--118},
  title        = {Configuration Spaces},
  volume       = {10},
}

@book{farbPrimerMappingClass2012,
  author    = {Farb, Benson and Margalit, Dan},
  date      = {2012},
  isbn      = {978-0-691-14794-9},
  mrnumber  = {2850125},
  number    = {49},
  pagetotal = {472},
  publisher = {Princeton University Press},
  series    = {Princeton {{Mathematical Series}}},
  title     = {A Primer on Mapping Class Groups},
}

@article{fuksCohomologiesGroupCOS1970,
  author       = {Fuks, D. B.},
  date         = {1970},
  doi          = {10.1007/BF01094491},
  issn         = {1573-8485},
  journaltitle = {Funct. Anal. Appl.},
  langid       = {english},
  number       = {2},
  pages        = {143--151},
  title        = {Cohomologies of the Group {{COS}} Mod 2},
  volume       = {4},
}

@article{hainautRepresentationAsymptotics2025,
  author     = {Hainaut, Louis and Knudsen, Ben and Wawrykow, Nicholas},
  date       = {2025},
  eprint     = {2510.00201},
  eprinttype = {arXiv},
  title      = {Representation Asymptotics in the Homology of Pure Graph Braid Groups},
}

@article{humphreysRepresentationsSl21975,
  author       = {Humphreys, J. E.},
  date         = {1975},
  doi          = {10.1080/00029890.1975.11993765},
  issn         = {0002-9890},
  journaltitle = {Am. Math. Mon.},
  number       = {1},
  pages        = {21--39},
  publisher    = {Taylor \& Francis},
  title        = {Representations of {$SL(2,p)$}},
  volume       = {82},
}

@article{idrissiRocaPointSet2026,
  author     = {Idrissi, Najib and Roca i Lucio, Victor},
  date       = {2026},
  eprint     = {2606.26802},
  eprinttype = {arXiv},
  title      = {Homology of Configuration Spaces in Positive Characteristic via Point-Set Constructions},
}

@software{idrissiRocaLucioSoftware,
  author  = {Idrissi, Najib and Roca i Lucio, Victor},
  date    = {2026-06-25},
  doi     = {10.5281/zenodo.20843322},
  title   = {\texttt{uconf} -- SageMath library for computations with configuration spaces and operadic constructions},
  url     = {https://github.com/nidrissi/uconf},
  version = {1.0.0},
}

@article{jordanQuantumDModules2009,
  author       = {Jordan, David},
  date         = {2009},
  doi          = {10.1093/imrp/rnp012},
  eprint       = {0805.2766},
  eprinttype   = {arXiv},
  journaltitle = {Int. Math. Res. Not. IMRN},
  number       = {11},
  pages        = {2081--2105},
  title        = {Quantum {$D$}-Modules, Elliptic Braid Groups, and Double Affine Hecke Algebras},
}

@article{knudsenBettiNumbersStability2017,
  author       = {Knudsen, Ben},
  date         = {2017},
  doi          = {10.2140/agt.2017.17.3137},
  eprint       = {1405.6696},
  eprinttype   = {arXiv},
  issn         = {1472-2747},
  journaltitle = {Algebr. Geom. Topol.},
  number       = {5},
  pages        = {3137--3187},
  title        = {Betti Numbers and Stability for Configuration Spaces via Factorization Homology},
  volume       = {17},
}

@article{koCharacteristicsGraphBraid2012,
  author       = {Ko, Ki Hyoung and Park, Hyo Won},
  date         = {2012},
  doi          = {10.1007/s00454-012-9459-8},
  eprint       = {1101.2648},
  eprinttype   = {arXiv},
  issn         = {0179-5376},
  journaltitle = {Discrete Comput. Geom.},
  number       = {4},
  pages        = {915--963},
  title        = {Characteristics of Graph Braid Groups},
  volume       = {48},
}

@article{maciazekNonAbelianQuantum2019,
  author       = {Maciążek, Tomasz and Sawicki, Adam},
  date         = {2019},
  doi          = {10.1007/s00220-019-03583-5},
  eprint       = {1806.02846},
  eprinttype   = {arXiv},
  issn         = {0010-3616},
  journaltitle = {Commun. Math. Phys.},
  number       = {3},
  pages        = {921--973},
  title        = {Non-Abelian Quantum Statistics on Graphs},
  volume       = {371},
}

@article{maguireComputingCohomology2016,
  author     = {Maguire, Megan},
  date       = {2016},
  eprint     = {1612.06314},
  eprinttype = {arXiv},
  note       = {With an appendix by Matthew Christie and Derek Francour},
  title      = {Computing Cohomology of Configuration Spaces},
}

@article{mcduffConfigurationSpacesPositive1975,
  author       = {McDuff, Dusa},
  date         = {1975},
  doi          = {10.1016/0040-9383(75)90038-5},
  issn         = {0040-9383},
  journaltitle = {Topology},
  number       = {1},
  pages        = {91--107},
  title        = {Configuration Spaces of Positive and Negative Particles},
  volume       = {14},
}

@article{napolitanoCohomologyConfigurationSpaces2003,
  author       = {Napolitano, Fabien},
  date         = {2003},
  doi          = {10.1112/S0024610703004617},
  issn         = {0024-6107},
  journaltitle = {J. Lond. Math. Soc.},
  number       = {2},
  pages        = {477--492},
  series       = {2},
  shortjournal = {J. Lond. Math. Soc.},
  title        = {On the Cohomology of Configuration Spaces on Surfaces},
  volume       = {68},
}

@article{pagariaCohomologyRingsTorus2020,
  author       = {Pagaria, Roberto},
  date         = {2020},
  doi          = {10.2140/agt.2020.20.2995},
  eprint       = {1901.01171},
  eprinttype   = {arXiv},
  journaltitle = {Algebr. Geom. Topol.},
  number       = {6},
  pages        = {2995--3012},
  title        = {The Cohomology Rings of the Unordered Configuration Spaces of the Torus},
  volume       = {20},
}

@article{perronFiltrationJohnsonGroupe2008,
  author       = {Perron, Bernard},
  date         = {2008},
  doi          = {10.1016/j.crma.2008.04.015},
  issn         = {1631-073X},
  journaltitle = {C. R. Math. Acad. Sci. Paris},
  number       = {11--12},
  pages        = {667--670},
  title        = {Filtration de {Johnson} et groupe de {Torelli} modulo {$p$}, {$p$} premier},
  volume       = {346},
}

@article{randalWilliamsHomologicalStability2013,
  author       = {Randal-Williams, Oscar},
  date         = {2013},
  doi          = {10.1093/qmath/har033},
  eprint       = {1105.5257},
  eprinttype   = {arXiv},
  issn         = {0033-5606},
  journaltitle = {Q. J. Math.},
  number       = {1},
  pages        = {303--326},
  title        = {Homological Stability for Unordered Configuration Spaces},
  volume       = {64},
}

@article{schiesslBettiNumbersTorus2016,
  author     = {Schiessl, Christoph},
  date       = {2016},
  eprint     = {1602.04748},
  eprinttype = {arXiv},
  title      = {Betti Numbers of Unordered Configuration Spaces of the Torus},
}

@book{serreLinearRepresentations1977,
  author     = {Serre, Jean-Pierre},
  date       = {1977},
  doi        = {10.1007/978-1-4684-9458-7},
  mrnumber   = {0450380},
  number     = {42},
  publisher  = {Springer-Verlag},
  series     = {Graduate Texts in Mathematics},
  title      = {Linear Representations of Finite Groups},
  translator = {Scott, Leonard L.},
}

@article{vainsteinCohomologyBraidGroups1978,
  author       = {Vaĭnšteĭn, F. V.},
  date         = {1978},
  issn         = {0374-1990},
  journaltitle = {Funkts. Anal. Prilozh.},
  mrnumber     = {498903},
  number       = {2},
  pages        = {72--73},
  title        = {The Cohomology of Braid Groups},
  volume       = {12},
}

@article{zhangQuillenHomology2025,
  author       = {Zhang, Adela YiYu},
  date         = {2025},
  doi          = {10.2140/agt.2025.25.1945},
  eprint       = {2110.08428},
  eprinttype   = {arXiv},
  journaltitle = {Algebr. Geom. Topol.},
  number       = {4},
  pages        = {1945--1997},
  title        = {Quillen Homology of Spectral Lie Algebras with Application to Mod {$p$} Homology of Labeled Configuration Spaces},
  volume       = {25},
}
\end{document}